\documentclass[a4paper]{amsart}

\usepackage[utf8]{inputenc}
\usepackage[T1]{fontenc}
\usepackage{lmodern, enumerate}
\usepackage{amssymb,amsxtra}
\usepackage[all]{xy}
\usepackage{nicefrac,mathtools}
\usepackage{microtype}
\usepackage{amscd}
\usepackage{xcolor}
\usepackage[pdftitle={Exotic Baum-Connes conjecture},
 pdfauthor={Alcides Buss, Siegfried Echterhoff, Rufus Willet},
 pdfsubject={Mathematics}
]{hyperref}
\usepackage[lite]{amsrefs}
\newcommand*{\MRref}[2]{ \href{http://www.ams.org/mathscinet-getitem?mr=#1}{MR #1}}
\newcommand*{\arxiv}[1]{\href{http://www.arxiv.org/abs/#1}{arXiv: #1}}

\newcommand*{\magc}[1]{{\color{magenta}#1}}

\numberwithin{equation}{section}
\theoremstyle{plain}
\newtheorem{theorem}[equation]{Theorem}
\newtheorem{lemma}[equation]{Lemma}
\newtheorem{proposition}[equation]{Proposition}
\newtheorem{corollary}[equation]{Corollary}
\theoremstyle{definition}
\newtheorem{definition}[equation]{Definition}
\newtheorem{convention}[equation]{Convention}
\theoremstyle{remark}
\newtheorem{remark}[equation]{Remark}
\newtheorem{example}[equation]{Example}

\newtheorem{question}[equation]{Question}

\DeclareMathOperator{\Aut}{Aut}% automorphism group
\DeclareMathOperator{\cspn}{\overline{span}}
\DeclareMathOperator{\spn}{span}
\DeclareMathOperator{\Ind}{Ind}

\DeclareMathOperator{\Mor}{\mathrm{Mor}}
\DeclareMathOperator{\ext}{\mathrm{ext}}

\renewcommand{\top}{{\operatorname{top}}}

\renewcommand{\H}{\mathcal H}
\newcommand*{\nb}{\nobreakdash}
\newcommand*{\Star}{\(^*\)\nobreakdash-}
\newcommand*{\dd}{\,d}

\newcommand{\cB}{\mathcal L} % -> changed notation here to be consistent with our other use: \Lb below...

\newcommand*{\C}{\mathbb C}

\newcommand*{\N}{\mathbb N}

\newcommand*{\Lb}{\mathcal L}%adjointable operators on a Hilbert module
\newcommand*{\K}{\mathcal K}% Compact operators
\newcommand*{\M}{\mathcal{M}} % multiplier algebra
\newcommand*{\ZM}{\mathcal{ZM}}

\newcommand*{\cont}{C}% continuous functions
\newcommand*{\contz}{\cont_0}%continuous functions vanishing at infinity
\newcommand*{\contc}{\cont_c}%continuous functions with compact support
\newcommand*{\id}{\textup{id}}%identity
\newcommand*{\Ad}{\textup{Ad}}%conjugation by a unitary

\newcommand*{\U}{\mathcal U}% unitaries
\newcommand*{\E}{\mathcal E}% Hilbert modules
\newcommand*{\defeq}{\mathrel{\vcentcolon=}}% used for definitions
\newcommand*{\congto}{\xrightarrow\sim}

\newcommand*{\braket}[2]{\langle#1\!\mid\!#2\rangle}

\newcommand*{\sbe}{\subseteq} % inclusion
\newcommand*{\F}{\mathcal F}
\newcommand*{\cstar}{\texorpdfstring{$C^*$\nobreakdash-\hspace{0pt}}{*-}}
\newcommand*{\into}{\hookrightarrow}
\newcommand*{\onto}{\twoheadrightarrow}
\newcommand*{\red}{\mathrm{r}}
\renewcommand*{\max}{\mathrm{max}}
\newcommand*{\un}{\mathrm{u}}

\newcommand*{\Cor}{\mathfrak{Corr}}% category of C*-algebras with homomorphisms for trivial groups
\newcommand{\as}{\operatorname{as}}

\newcommand{\Res}{\operatorname{Res}}

\newcommand{\lk}{\langle}
\newcommand{\rk}{\rangle}
\newcommand{\Rep}{{\operatorname{Rep}}}
\newcommand{\dual}{\widehat}
\newcommand{\dualG}{\widehat{G}}
\newcommand{\Manoa}{M\=anoa}

\newcommand{\Hawaii}{Hawai\kern.05em`\kern.05em\relax i}

\begin{document}
\title[Exotic crossed products]{Exotic crossed products}

\author{Alcides Buss}
\email{alcides@mtm.ufsc.br}
\address{Departamento de Matem\'atica\\
 Universidade Federal de Santa Catarina\\
 88.040-900 Florian\'opolis-SC\\
 Brazil}

\author{Siegfried Echterhoff}
\email{echters@math.uni-muenster.de}
\address{Mathematisches Institut\\
 Westf\"alische Wilhelms-Universit\"at M\"un\-ster\\
 Einsteinstr.\ 62\\
 48149 M\"unster\\
 Germany}

\author{Rufus Willett}
\email{rufus@math.hawaii.edu}
\address{Mathematics Department\\
 University of \Hawaii~at \Manoa\\
Keller 401A \\
2565 McCarthy Mall \\
 Honolulu\\
 HI 96822\\
USA}

\begin{abstract}
An exotic crossed product is a way of associating a \cstar{}algebra to each \cstar{}dynamical system that generalizes the well-known universal and reduced crossed products.  Exotic crossed products provide natural generalizations of, and tools to study, exotic group \cstar{}algebras as recently considered by Brown-Guentner and others.  They also form an essential part of a recent program to reformulate the Baum-Connes conjecture with coefficients so as to mollify the counterexamples caused by failures of exactness.

In this paper, we survey some constructions of exotic group algebras and exotic crossed products.  Summarising our earlier work, we single out a large class of crossed products --- the \emph{correspondence functors} --- that have many properties known for the maximal and reduced crossed products: for example, they extend to categories of equivariant correspondences, and have a compatible descent morphism in $KK$-theory.   Combined with known results on $K$-amenability and the Baum-Connes conjecture, this allows us to compute the $K$-theory of many exotic group algebras.  It also gives new information about the reformulation of the Baum-Connes Conjecture mentioned above.  Finally, we present some new results relating exotic crossed products for a group and its closed subgroups, and discuss connections with the reformulated Baum-Connes conjecture.

\end{abstract}


\thanks{Supported by Deutsche Forschungsgemeinschaft  (SFB 878, Groups, Geometry \& Actions), by CNPq/CAPES -- Brazil, and by the US NSF (DMS 1401126).}


\maketitle

\section{Introduction}

Given a C*-dynamical system $(A,G,\alpha)$, there are two classical ways to assign a C*-algebra to it which reflect important properties of the
 system: the universal (or maximal) crossed product $A\rtimes_{\alpha,\un}G$, which is universal for covariant representations,  and the reduced crossed product $A\rtimes_{\alpha,\red}G$ defined as the image of  $A\rtimes_{\alpha,\un}G$ under the regular representation. If $A=\C$ we recover the maximal and reduced group C*-algebras
$C_\un^*(G)$ and $C_\red^*(G)$ of the group $G$.

Recently there has been a growing interest in studying
more general ``exotic'' group algebras $C_\mu^*(G)$ and crossed products $A\rtimes_{\alpha,\mu}G$ which are
defined as completions of the convolution algebra $C_c(G)$ (resp. $C_c(G,A)$) with respect
to C*-norms which lie between the full and reduced norms $\|\cdot\|_\un$ and $\|\cdot\|_\red$.  For exotic group algebras, this started with the work \cite{Brown-Guentner:New_completions} of Brown and Guentner: given a discrete group $G$ and an (algebraic) ideal
$E$ in $l^\infty(G)$, Brown and Guentner assign an exotic group algebra $C^*_{E,BG}(G)$ by looking at representations of $G$ that have `enough' matrix coefficients in
$E$.  Later work of Okayasu (\cite{Okayasu:Free-group}) shows that
this theory is extremely rich: If $G$ is a discrete group which contains the free group $F_2$ on two generators
as a subgroup, then there are uncountably many different exotic group algebras $C_{E}^*(G)$ of this type.  Wiersma (\cite{Wiersma})
subsequently extended the framework of Brown and Guentner to general locally compact groups, and showed an analogous result to Okayasu's for the connected group $SL(2,\mathbb R)$.  Motivated by the work of Brown and Guentner, the paper \cite{Kaliszewski-Landstad-Quigg:Exotic-coactions} of Kaliszewski, Landstad and Quigg
 introduced a second construction of exotic group algebras based on the dual pairing of $C^*_\un(G)$ with the Fourier-Stieltjes algebra $B(G)$, and studied the connection to coactions.
We shall give a detailed introduction to exotic group algebras in \S2 below and we use this opportunity
to clarify the differences between the approaches of Brown-Guentner and of Kaliszewski-Landstad-Quigg.

Work on exotic crossed products starts also with Brown-Guentner (\cite{Brown-Guentner:New_completions}) and Kaliszewski-Landstad-Quigg (\cite{Kaliszewski-Landstad-Quigg:Exotic-coactions}) who both indicated ways to associate exotic crossed-products $(A,\alpha)\mapsto A\rtimes_{\alpha, \mu_E}G$
to appropriate spaces of matrix coefficients $E$. The constructions of exotic group algebras of Brown-Guentner and of Kaliszewski-Landstad-Quigg are closely related, but their constructions of exotic crossed products are fundamentally quite different in nature; nonetheless, it was observed later (\cite{Buss-Echterhoff:Exotic_GFPA,Baum-Guentner-Willett}) that both constructions are functorial in the sense that
every $G$-equivariant \Star{}homomorphism $\Phi:A\to B$  between two $G$-algebras $(A,\alpha)$ and $(B,\beta)$
descends to a \Star{}homomorphism $\Phi\rtimes_{\mu_E}G:A\rtimes_{\alpha,\mu_E}G\to B\rtimes_{\beta,\mu_E}G$ in a canonical way.

Such exotic crossed product \emph{functors} were formally introduced by Baum, Guentner, and Willett in \cite{Baum-Guentner-Willett}
as part of a program to ``fix'' the Baum-Connes conjecture for computing the $K$-theory of crossed products. The original conjecture
(in the general version with coefficients) claimed that for each system $(A,G,\alpha)$ a certain {\em assembly map}
$$\as_{(G,A)}^\red: K_*^{\operatorname{top}}(G;A)\to K_*(A\rtimes_{\alpha,\red}G)$$
should always be an isomorphism, where $K_*^{\operatorname{top}}(G;A)$ is the {\em topological $K$-theory of $G$
with coefficient $A$} (we refer to \cite{Baum-Connes-Higson:BC} for the construction of this group and for a discussion of the
far reaching consequences for groups $G$ that satisfy the conjecture).
 Unfortunately, the conjecture is not true for all groups: the known counterexamples are all related to the failure
of $G$ being exact in the sense of Kirchberg and Wassermann.
In order to fix the conjecture, the main idea of
Baum, Guentner, and Willett was to replace the reduced crossed-product functor by the {\em smallest exact Morita compatible} crossed-product functor $(A,\alpha)\mapsto A\rtimes_{\alpha,\E}G$. Here Morita compatibility basically means that
the construction preserves {\em stabilisation} which is needed to construct an $E$-theory descent used
 for the construction of a direct assembly map
$$\as_{(G,A)}^\E: K_*^{\operatorname{top}}(G;A)\to K_*(A\rtimes_{\alpha,\E}G);$$
exactness and Morita compatibility are both predicted on the level of $K$-theory by the conjecture, so are natural assumptions.
The new conjecture asserts that this  map is always an isomorphism. In fact, Baum, Guentner, and Willett show
in \cite{Baum-Guentner-Willett} that several of the known counterexamples for the original conjecture satisfy
the new conjecture and, so far, there are no known counterexamples for the reformulated conjecture.
\medskip

Motivated by the above described developments, the authors of this paper started in
\cite{Buss-Echterhoff-Willett} a more systematic study
of various functorial properties of exotic crossed-product functors $(A,\alpha)\mapsto A\rtimes_{\alpha,\mu}G$.  Recall first that a correspondence between two
C*-algebras  $A$ and $B$ consists of a Hilbert $B$-module $\E_B$ equipped with a left action $\Phi:A\to \mathcal L(\E_B)$ of
$A$ on $\E_B$.
In case of the universal or reduced crossed-products $\rtimes_\un$ or $\rtimes_\red$, respectively,
it has been known for a long time that they are
\begin{enumerate}
\item functorial for {\em generalised homomorphisms} in the sense that any $G$-equivariant \Star{}homomorphism
$\Phi:A\to \M(B)$ canonically  induces \Star{}homomorphisms
$\Phi\rtimes_\mu G: A\rtimes_{\alpha,\mu}G\to\M(B\rtimes_{\beta,\mu}G)$, $\mu=\un,\red$.
\item functorial for correspondences: This means that any $G$-equivariant \mbox{$A-B$} correspondence $\E_B$ admits a {\em descent}
$\E_B\rtimes_\mu G$ which is an $A\rtimes_{\alpha,\mu}G-B\rtimes_{\beta,\mu}G$  correspondence and  this construction
preserves composition of correspondences.
\item functorial for Kasparov's $KK$-theory: This simply means that there is a descent homomorphism
$$j_G^\mu: KK^G(A,B)\to KK(A\rtimes_{\alpha,\mu}G, B\rtimes_{\beta,\mu}G), \quad\mu=\un,\red,$$
which is compatible with taking Kasparov products.
\end{enumerate}

In \S\S 3--5 we give a survey of the general theory of exotic crossed products and, in particular,  the
results obtained  in \cite{Buss-Echterhoff-Willett} in which we give characterisations
of  crossed products which enjoy strong functorial properties as described above.
As a sample,  it turned out that functoriality for generalised homomorphisms is equivalent to the  {\em ideal property}
which asserts that for each $G$-invariant ideal $I\subseteq G$  the inclusion $\iota:I\into A$ descends
to a faithful morphism $\iota\rtimes_{\mu}G:I\rtimes_\mu G\into A\rtimes_\mu G$. Similarly, functoriality for correspondences
is equivalent to the  {\em projection property} which asserts that for every $G$-invariant
projection $p\in \M(A)$ the inclusion $\iota:pAp\into A$ descends to an inclusion
 $\iota\rtimes_{\mu}G: pAp\rtimes_\mu G\into A\rtimes_\mu G$. Another useful characterisation of
  correspondence functors is functoriality with respect to $G$-equivariant completely positive maps, and as an application of this
 we show here as a new result that correspondence crossed products behave well with respect to taking tensor products with
 nuclear C*-algebras:  If $\rtimes_\mu$ is a correspondence functor, then
$$(A\otimes B)\rtimes_{\alpha\otimes \id_B,\mu}G\cong (A\rtimes_{\alpha,\mu}G)\otimes B$$
 for any nuclear C*-algebra $B$. As an application, we can show that under suitable assumptions it turns
 out that any crossed-product functor which is faithful in the sense that it does not send nonzero objects
 to $\{0\}$ must dominate the reduced crossed product functor. This motivates the requirement that exotic crossed products
lie between the maximal and the reduced ones.

It is straightforward to see that the functors constructed by Brown-Guentner and Kaliszewski-Landstad-Quigg (which we shall call {\em BG-functors} or {\em KLQ-functors}, respectively)
satisfy the ideal property. Moreover, all KLQ-functors are correspondence functors, but a BG-functor is a correspondence functor if and only if it coincides with the universal crossed-product functor.

It was shown in \cite{Buss-Echterhoff-Willett} that every correspondence functor admits a descent in $KK$-theory.
Adapting ideas of Cuntz and
 Julg-Valette from \cite{Cuntz:Kamenable} and \cite{Julg-Valette:Kamenable} this fact has then been used to show that
 for any $K$-amenable group in the sense of Cuntz and for any correspondence functor $\rtimes_\mu$, the canonical morphisms
 $$A\rtimes_{\alpha,u}G\onto A\rtimes_{\alpha,\mu}G\onto A\rtimes_{\alpha,\red}G$$
 are $KK$-equivalences. In particular, this implies that for $G=F_n$, the free group in $n$ generators, $n\in \mathbb N\cup \{\infty\}$,
all of the uncountably many different exotic group algebras $C^*_{E}(F_n)$ corresponding to nontrivial $G$-invariant ideals
$E\subseteq B(G)$ are $KK$-equivalent. The result would be false without the assumption that $E$ is an ideal, and indeed it is crucial for us that an exotic group algebra is of the form $\C\rtimes_\mu G$ for a correspondence functor $\rtimes_\mu$ if and only if it is of the form $C^*_E(G)$ for a $G$-invariant weak$^*$-closed ideal $E$ in $B(G)$.

 We also report on the result that the minimal exact Morita compatible
functor $\rtimes_{\E}$ coincides, at least on separable systems, with the minimal exact correspondence functor $\rtimes_{\E_\Cor}$,
whose existence is shown in \cite{Buss-Echterhoff-Willett}. Hence our results allow one to use the full force of Kasparov's
equivariant $KK$-theory for the study of the reformulated Baum-Connes conjecture.

In the remaining sections (\S\S6--8) we report on some new results about the relation of crossed-product functors of
a group $G$ with crossed-product functors of closed subgroups $H$ of $G$. Based on Green's imprimitivity theorem
we give a procedure of restricting  crossed-product functors from $G$ to $H$.
Conversely, given a crossed-product functor $\rtimes_{\nu}$  for $H$, we describe two different ways to
assign a crossed-product functor for $G$ to it: one is the induced functor $\rtimes_{\Ind\nu}$ and  one is the extended
functor $\rtimes_{\ext\nu}$. All these constructions preserve the ideal property (hence functoriality for generalised morphisms)
and send correspondence functors to correspondence functors. Restriction and extension preserve exactness
in general, but this can fail for induction.
After introducing the restriction, induction and extension procedures in \S 6 we show in \S 7 that the
restriction of the minimal exact correspondence functor $\rtimes_{\E^G_\Cor}$ of $G$ to a closed {\em normal}
subgroup $N$ of $G$ always coincides with the minimal exact correspondence functor $\rtimes_{\E^N_\Cor}$ of $N$.
Using results of Chabert and Echterhoff from \cite{Chabert-Echterhoff:Permanence} this implies that the
validity of the reformulated Baum-Connes conjecture for $G$ always passes to closed normal subgroups $N\subseteq G$.
Another result in this direction (obtained in \S 6), shows that for a cocompact closed subgroup $H\subseteq G$
the induction of $\rtimes_{\E^H_\Cor}$ to $G$ coincides with  $\rtimes_{\E^G_\Cor}$.

Finally, in \S 8, we close this paper with some questions and remarks related to induction, restriction and extension
in connection with more general permanence properties of the reformulated Baum-Connes conjecture. Unfortunately, so far it seems that we have more questions than solutions!
\medskip

Part of the work on this paper took place during a visit of the first and third authors to the Westf\"{a}lische Wilhelms-Universit\"{a}t, M\"{u}nster.  We would like to thank this institution for its hospitality.

\label{sec:introduction}\label{ex alg sec}

\section{Exotic group algebras}\label{sec:preliminaries}

Let $G$ be a locally compact group equipped with a fixed left invariant Haar measure.  We are interested in \cstar{}algebras connected to strongly continuous unitary representations $u:G\to \U(\H)$ of $G$, which we call simply \emph{representations} of $G$.

Let $C_c(G)$ denote the space of continuous, compactly supported, complex-valued functions on $G$, which is a \Star{}algebra when equipped with the product and involution
$$
f*g(t)\defeq\int_G f(s)g(s^{-1}t)\dd{s}\quad\mbox{and}\quad f^*(t)\defeq \Delta(t)^{-1}\overline{f(t^{-1})},
$$
where $\Delta:G\to (0,\infty)$ is the modular function.  Any representation $u:G\to \U(\H)$ integrates to a \Star{}representation $\tilde{u}:C_c(G)\to \cB(\H)$ defined by
$$
\tilde{u}(f)\defeq \int_G f(s)u_s \dd{s}.
$$
In what follows we shall usually omit the tilde in our notations, thus using the same notation for the unitary representation $u$ and its integrated form.

The following two completions of $C_c(G)$ are intimately tied to the representation theory of $G$, and for this and other reasons have been widely studied.

\begin{definition}\label{un red}
The \emph{universal group algebra} $C^*_\un(G)$ is the completion of $C_c(G)$ with respect to the norm
$$
\|f\|_\un\defeq \sup\{\|u(f)\|:u \text{ a representation of $G$}\}.
$$
Let $\lambda:G\to \U(L^2(G))$ be the left regular representation: $\lambda_s(\xi)(t)\defeq \xi(s^{-1}t)$.  The \emph{reduced group algebra} $C^*_\red(G)$ is the completion of $C_c(G)$ with respect to the norm
$$
\|f\|_\red\defeq \|\lambda(f)\|.
$$
\end{definition}

\begin{remark}\label{uni prop}
As every \cstar{}algebra $A$ admits a faithful nondegenerate representation on some Hilbert space,
every strictly continuous homomorphism $u\colon G\to \U\M(A)$ integrates uniquely to a nondegenerate \Star{}homomorphism $C^*_\un(G)\to \M(A)$.
Conversely, every such homomorphism is the integrated form of a representation of $G$ and this universal property characterises $C^*_\un(G)$ up to isomorphism.
\end{remark}

\begin{definition}\label{ex gp alg}
A \emph{group algebra} $C_\mu^*(G)$  is any \cstar{}algebra completion of $C_c(G)$ for a norm  $\|\cdot\|_\mu$
such that for all $f\in C_c(G)$
$$
\|f\|_\un \geq \|f\|_\mu\geq \|f\|_\red.
$$
A group algebra $C^*_\mu(G)$ is \emph{exotic} if $\|\cdot\|_\mu$ is not equal to either the maximal or reduced norms.
\end{definition}

It is perhaps not immediately clear that interesting examples exist!  There have been two recent approaches to this in the literature that we discuss below, the first due to Brown and Guentner
\cite{Brown-Guentner:New_completions}, and the second due to Kaliszewski, Landstad, and Quigg \cite[Section 3]{Kaliszewski-Landstad-Quigg:Exotic}.  We first discuss the construction of Brown-Guentner, and then describe theorems of Okayasu and Wiersma that show that the theory of exotic group algebras is very rich.

\begin{definition}\label{bg def}
Let $D$ be a set of complex-valued functions on $G$.  A representation $u:G\to\U(\H)$ of $G$ is called a \emph{$D$-representation} if there exists a dense subspace $\H_0$ of $\H$ such that for all $\xi,\eta\in \H_0$, the \emph{matrix coefficient}
$$
G\to\C,\quad g\mapsto \langle \xi,u_g\eta\rangle
$$
is in $D$.\footnote{We always assume inner products on Hilbert spaces to be linear in the {\em second} variable.}  Define a \cstar{}seminorm on $C_c(G)$ by
$$
\|f\|_{D,BG}:=\sup\{\|u(f)\|~|~u \text{ a $D$-representation}\}.
$$
Define $C^*_{D,BG}(G)$ to be the Hausdorff completion of $C_c(G)$ for the seminorm $\|~\cdot~\|_{D,BG}$.
\end{definition}

\begin{remark}\label{bg hist}
\begin{enumerate}
\item In the original version of their definition \cite[Definition 2.2]{Brown-Guentner:New_completions}, Brown and Guentner work only with discrete groups, but their definition extends in an obvious way to all locally compact groups as already pointed out by Wiersma \cite[Section 3]{Wiersma}.  They moreover assume that $D$ is an algebraic ideal in the space $l^\infty(G)$; this is important for their applications.  It is clear, however, that the definition makes sense without these additional assumptions; we left them out as the extra generality seems harmless, and is occasionally useful.
\item Let $B(G)$ denote the \emph{Fourier-Stieltjes algebra} of $G$: the space of all matrix coefficients of $G$, which turns out to be a \Star{}algebra of bounded functions on $G$ for the natural pointwise operations;
we will have more to say on this below.
The \cstar{}algebra $C^*_{D,BG}(G)$ clearly only depends on the intersection $D\cap B(G)$.  It is often convenient to pass to this intersection as it is the algebraic and topological properties of $D\cap B(G)$ as a subset of the algebra $B(G)$ that are really relevant for the structure of $C^*_{D,BG}(G)$.  For example, when Brown and Guentner use that $D$ is an ideal in $l^\infty(G)$ for certain of their applications, what is really relevant is that $D\cap B(G)$ is an ideal in $B(G)$; the latter is a strictly weaker property. On the other hand, the passage from $l^\infty(G)$ to $B(G)$ also loses some concreteness, as it is generally difficult to tell when a given function is in $B(G)$.
\item The completion $C^*_{D,BG}(G)$ is a group \cstar{}algebra in our sense if and only if $\|\cdot\|_{D,BG}$ dominates the reduced norm.  For this, it is sufficient (for example) that $D$ contains all compactly supported functions in $B(G)$.
\end{enumerate}
\end{remark}

We now turn to a natural class of examples based on decay of matrix coefficients.  For $p\in [1,\infty)$, write $C^*_p(G)$ as shorthand for $C^*_{L^p(G),BG}(G)$.  Note that if $p<q$ then $L^p(G)\cap B(G)\subseteq L^q(G)\cap B(G)$ and thus the identity map on $C_c(G)$ extends to a surjective \Star{}homomorphism
\begin{equation}\label{pq quotient}
C^*_q(G)\onto C^*_p(G).
\end{equation}
The following theorem is due to Okayasu for $F_2$ \cite{Okayasu:Free-group} and Wiersma for $SL(2,\mathbb{R})$ \cite{Wiersma}: it shows in particular that the construction above gives rise to many interesting exotic group algebras.

\begin{theorem}\label{lp diff}
Say $G$ is either the free group $F_2$ on two generators, or $SL(2,\mathbb{R})$.  Let $p>q\geq 2$.  Then the canonical quotient map in line \eqref{pq quotient} above is not injective.
\end{theorem}

\begin{remark}
\begin{enumerate}
\item All the \cstar{}algebras $C^*_p(SL(2,\mathbb{R}))$ for $p\in (2,\infty)$ are in fact \emph{abstractly} isomorphic, as one can see by combining Wiersma's analysis \cite{Wiersma} with Mili\v{c}i\'{c}'s description of the structure of $C^*_\un(SL(2,\mathbb{R}))$ \cite{Milicic}.
    We do not know if $C^*_p(F_2)$ is abstractly isomorphic to $C^*_q(F_2)$ for any distinct $p,q$ in $(2,\infty)$.
\item The key ingredients in the proof of Theorem \ref{lp diff} are quite deep facts from harmonic analysis: for $F_2$ the result relies on Haagerup's study of $C^*_\red(F_2)$ in \cite{Haagerup}, while for $SL(2,\mathbb{R})$ the proof relies on aspects of the Kunze-Stein phenomenon \cite{Kunze-Stein}.
\item The result of Okayasu extends to all (non-amenable) discrete groups that contain $F_2$ as a subgroup, and conceivably to all non-amenable discrete groups.  However, Wiersma's result certainly does not extend to all non-amenable connected groups.  Indeed for any $n> 2$, it follows from results of Scaramuzzi \cite[Theorem III.3.3]{Scaramuzzi} (see also the discussion in \cite[Section V.3.3]{Howe}) that if $G=SL(n,\mathbb{R})$, then $C^*_p(G)=C^*_q(G)$ for all $p,q\in (n,\infty)$.
\end{enumerate}
\end{remark}

Having established that many interesting exotic group algebras exist, let us turn to the construction of Kaliszeswki-Landstad-Quigg.  This gives another perspective that will be especially convenient when we come to discuss exotic crossed products.
In order to do this, we need a little more background on the Fourier-Stieltjes algebra $B(G)$.
Recall that $B(G)$ consists of all matrix coefficients of $G$, that is, functions $\phi:G\to\C$ of the form $\phi(g)=\langle \xi,u_g\eta\rangle$ for some representation $u:G\to\U(\H)$ and vectors $\xi,\eta\in \H$; note that elements of $B(G)$ are necessarily bounded and continuous.  Straightforward algebraic checks based on the fact that one can take direct sums, contragredients, and tensor products of representations show $B(G)$ is a \Star{}algebra under the usual pointwise operations.  Note also that $B(G)$ is invariant under the $G$-actions on functions induced by the left and right translation actions of $G$ on itself.  For brevity, we say that a collection $E$ of functions on $G$ is \emph{translation invariant} if it is preserved by the actions on functions induced by the left and right translation actions of $G$ on itself.

There is a pairing between $C_c(G)$ and $B(G)$ defined by
$$
\langle f,\phi\rangle \defeq\int_G f(s)\phi(s)\dd{s},\quad f\in C_c(G), \phi\in B(G).
$$
This clearly extends to a bilinear pairing between $C^*_\un(G)$ and $B(G)$.  Moreover, it follows from the identification of unitary representations of $G$ and \Star{}representations of $C^*_\un(G)$, together with the GNS construction, that this pairing identifies $B(G)$ with the dual space $C_\un^*(G)'$ of $C^*_\un(G)$.  We equip $B(G)$ with the norm, and also weak* topology, coming from this identification.  See \cite{Eymard} for more information on $B(G)$.

Let now $\U(\M(C^*_\un(G)))$ be the unitary group of the multiplier algebra of $C^*_\un(G)$.  There is a \emph{universal representation} $u_G:G\to \U(\M(C^*_\un(G)))$ defined on elements of $C_c(G)$ by
\begin{equation}\label{uni rep}
(u_G(t)f)(s)=f(t^{-1}s),\quad (fu_G(t))(s)=f(st^{-1})\Delta(t^{-1}).
\end{equation}
The following lemma is straightforward to check directly; it is essentially the same as \cite[Lemma 3.1]{Kaliszewski-Landstad-Quigg:Exotic}.
\begin{lemma}\label{ti id}
Let $E$ be a subspace of $B(G)$, and $I\defeq \{a\in C^*_\un(G)~|~\langle a,\phi\rangle=0\text{ for all } \phi\in E\}$ be its pre-annihilator.  Then the following are equivalent:
\begin{enumerate}
\item $E$ is translation invariant;
\item $I$ is invariant under left and right multiplication by the image of $G$ under $u_G$;
\item $I$ is an ideal in $C_\un^*(G)$.
\end{enumerate}
\end{lemma}

Here then is the definition of Kaliszewski-Landstad-Quigg.

\begin{definition}\label{klq def}
Let $E$ be a translation invariant subspace of $B(G)$.  Let $I_E$ be the pre-annihilator of $E$, which is an ideal by Lemma \ref{ti id}.  Define a \cstar{}algebra
$$
C^*_{E,KLQ}(G)\defeq C^*_\un(G)/I_{E}.
$$
\end{definition}

This construction actually gives rise to all exotic group algebras (and indeed to all quotients of $C^*_\un(G)$ if we do not impose further restrictions on $E$).
To make this statement precise, recall that for us a group $C^*$-algebra is essentially the same thing as a $C^*$-algebra norm on $C_c(G)$ that dominates the reduced norm, and is dominated by the universal norm.  If $\|\cdot\|$ is such a norm, write $E_{\|\cdot\|}$ for the elements of $B(G)$ that are continuous for that norm; conversely, if $E$ is a translation invariant subspace of $B(G)$, write $\|\cdot\|_{E,KLQ}$ for the norm on $C^*_{E,KLQ}(G)$ restricted to $C_c(G)$.  Let $B_\red(G)$ denote the elements of $B(G)$ that are continuous for the reduced norm, so $B_\red(G)$ identifies canonically with $C^*_\red(G)'$.
It also coincides with the weak$^*$ closure of $B_c(G):=C_c(G)\cap B(G)$ in $B(G)$
by \cite[Lemma 3.9]{Kaliszewski-Landstad-Quigg:Exotic}.

\begin{proposition}\label{everything}
The assignments
$$
E\mapsto \|\cdot\|_{E,KLQ},\quad  \|\cdot\|\mapsto E_{\|\cdot\|}
$$
define mutually inverse bijections between the set of weak$^*$-closed, proper, translation invariant subspaces of $B(G)$ that strictly contain $B_\red(G)$, and the set of exotic group algebra norms on $C_c(G)$.
\end{proposition}

\begin{proof}
Recall first that if $X$ is any Banach space, then taking annihilators and pre-annihilators gives a bijective correspondence between closed subspaces of $X$ and weak$^*$-closed subspaces of the dual $X'$.  Specialising this to $X=C^*_\un(G)$ and using Lemma \ref{ti id} gives a bijective correspondence between translation invariant weak$^*$-closed subspaces of $X'=B(G)$ and closed ideals of $C^*_\un(G)$.  Specialising yet further gives a bijective correspondence between the weak$^*$-closed, proper, translation invariant subspaces of $B(G)$ that strictly contain $B_\red(G)$, and the closed nonzero ideals in $C^*_\un(G)$ that are strictly contained in the kernel of the canonical quotient map $C^*_\un(G)\to C^*_\red(G)$.  The latter are clearly in bijective correspondence with \cstar{}semi-norms on $C^*_\un(G)$ that restrict to exotic group algebra norms on $C_c(G)$ by associating an ideal to the corresponding quotient norm and vice versa.  Following these correspondences through gives the result.
\end{proof}

It is worth noticing that by  \cite{Kaliszewski-Landstad-Quigg:Exotic}*{Lemma~3.14} the condition $B_\red(G)\subseteq E$ is automatic, if $E$ is a weak$^*$ closed translation invariant
{\em ideal} in $B(G)$.

For a translation invariant space $D$ of functions on $G$ containing $B_c(G)=B(G)\cap \contc(G)$
it is asserted in \cite{Kaliszewski-Landstad-Quigg:Exotic}*{Lemma~3.5} that $C^*_{D,BG}(G)\cong C^*_{E,KLQ}(G)$ if we define $E\defeq D\cap B(G)$. However there seems to be a gap in the proof of \cite{Kaliszewski-Landstad-Quigg:Exotic}*{Lemma~3.5} and it is not clear to us whether the conclusion of that lemma holds.
So our next goal is to clarify the relationship between the Brown-Guentner and Kaliszewski-Landstad-Quigg constructions.  We denote by $P(G)$ the collection of continuous positive type functions on $G$, which is a cone in $B(G)$. The following proposition is closely related
to \cite[Propositions 4.1 and 4.3]{Wiersma}.

\begin{proposition}\label{same}
Say $D$ is a translation invariant subspace of the vector space of complex-valued functions on $G$.  Then $E\defeq \text{span}(D\cap P(G))$ is a translation invariant subspace of $B(G)$.
Moreover, the identity map on $C_c(G)$ extends to an isomorphism $C^*_{D,BG}(G)\cong C^*_{E,KLQ}(G)$.
\end{proposition}

\begin{proof}
As $P(G)$ is contained in $B(G)$, it is clear that $E$ is a subspace of $B(G)$; we must show that it is translation invariant.  We will focus on the left action; the case of the right action is similar.

For a function $\phi:G\to\C$ and $s\in G$, write $_s\phi$ for the left-translate defined by $_s\phi:t\mapsto \phi(s^{-1}t)$, and $\phi_s:t\mapsto \phi(ts)$ for the right translate.
It suffices to show that if $\phi\in P(G)\cap D$ and $s\in G$, then $_s\phi$ is in $E$.  As $\phi$ is in $P(G)$, we may write
$$
\phi(t)=\langle \xi,u_t\xi\rangle
$$
for some representation $u:G\to \U(\H)$ and $\xi\in \H$.  It follows that
$$
_s\phi(t)=\langle \xi,u_{s^{-1}t}\xi\rangle=\langle u_s\xi,u_t\xi\rangle=\frac{1}{4}\sum_{k=0}^3i^k \langle u_s\xi+i^k\xi,u_t(u_s\xi+i^k\xi)\rangle.
$$
Write $\phi_k(t)=\langle u_s\xi+i^k\xi,u_t(u_s\xi+i^k\xi)\rangle$; as each $\phi_k$ is in $P(G)$, it suffices to show that each is in $D$.  Note, however, that
\begin{align*}
\phi_k(t) & =\langle u_s\xi,u_tu_s\xi\rangle+i^k\langle u_s\xi,u_t\xi\rangle+i^{-k}\langle \xi,u_tu_s\xi\rangle+\langle \xi,u_t\xi\rangle \\
&= {_s\phi_s}(t)+i^k{}_s\phi(t)+i^{-k}\phi_s(t)+\phi(t),
\end{align*}
and so $\phi_k$ is in $D$ because $D$ is translation invariant.

To see that the identity map on $C_c(G)$ induces an isomorphism $C^*_{D,BG}(G)\cong C^*_{E,KLQ}(G)$, it suffices to check that
$$
I_{E}=\bigcap\{\text{kernel}(u:C^*_\un(G)\to \Lb(\H))~|~u:G\to \U(\H) \text{ a $D$-representation}\}
$$
Say first $a\in I_{E}$, and let $u:G\to\U(\H)$ be a $D$-representation with associated dense subspace $\H_0$ giving rise to a dense set of matrix coefficients in $D$.   Then for any $\xi\in \H_0$, if $\phi(s)=\langle \xi,u_s\xi\rangle$ then $\phi$ is in $D\cap P(G)\subseteq E$ and we have that
$$
\langle \xi,u(a)\xi\rangle=\langle a, \phi\rangle=0.
$$
As $\H_0$ is dense, this forces $\langle \eta,u(a)\eta\rangle=0$ for all $\eta\in \H$, and thus $a$ to be in the kernel of (the integrated form of) $u$.  Hence
$$
I_{E}\subseteq \bigcap\{\text{kernel}(u:C^*_\un(G)\to \Lb(\H))~|~u:G\to \U(\H) \text{ a $D$-representation}\}.
$$

For the converse inclusion, say $a\in C^*_\un(G)\setminus I_{E}$.  Then there is an element $\phi$ of $E$ such that $\langle a,\phi\rangle\neq 0$.  As $E$ is spanned by $D\cap P(G)$, we may assume moreover that $\phi$ is in $D\cap P(G)$.  Let $(u,\H,\xi)$ be the GNS triple associated to $\phi$, and note that $u:G\to\U(\H)$ is a $D$-representation: indeed, since $D$ is translation invariant we may take $\H_0=\text{span}\{u_s\xi~|~s\in G\}$.  Hence $\langle \xi,u(a)\xi\rangle=\langle a,\phi\rangle$, which is non-zero, and thus $a\not\in \text{kernel}(u:C^*_\un(G)\to \Lb(\H))$.
\end{proof}

\begin{corollary}\label{same same}
Let $E\subseteq B(G)$ be a translation invariant subspace of $B(G)$.  Then the following are equivalent:
\begin{enumerate}
\item  $E_0=\spn(E\cap P(G))$ is weak* dense in $E$.
\item The identity map on
$C_c(G)$ extends to an isomorphism $C^*_{E,BG}(G)\cong C^*_{E,KLQ}(G)$.
\end{enumerate}
\end{corollary}

\begin{proof}
It follows from the definition of KLQ-group algebras  together with  Proposition \ref{everything} that
 $C_{E_0, KLQ}^*(G)=C_{E,KLQ}^*(G)$ if and only if the weak$^*$ closures of $E_0$ and $E$ coincide.
 Combining this with  Proposition \ref{same} gives the equivalence of (1) and (2).
\end{proof}

At the time of writing we do not know whether $\spn(E\cap P(G))$ is weak$^*$ dense in
$E$ for all translation invariant subspaces of $B(G)$, but it seems likely that there are counter examples.
The following lemma gives some results:

\begin{lemma}\label{lem-closed}
Suppose $E\subseteq B(G)$ is a translation invariant subspace which satisfies one of the following
conditions:
\begin{enumerate}
\item $E$ is closed with respect to the supremum-norm on $B(G)\subseteq C_b(G)$.
\item $E$ is closed in the norm on $B(G)$ coming from identifying $B(G)$ with the dual space $C_\un^*(G)'$.
\item $E$ is weak* closed.
\end{enumerate}
Then the identity map on $C_c(G)$ extends to an isomorphism $C_{E,BG}^*(G)=C_{E,KLQ}^*(G)$.
\end{lemma}

\begin{proof}
We show that in all three cases every element in $E$ can be written as a linear combination
of positive elements in $E$. Having done this, all three cases then follow from Proposition \ref{same}.
For this let $s\mapsto \phi(s)=\langle \xi, u_s\eta\rangle$ be a nonzero element of $E$ for
some unitary representation $u:G\to\U(\H)$.
By passing to $\H_0=\cspn(u(G)\xi)\cap \cspn(u(G)\eta)$ and the images of $\xi,\eta$ under the
orthogonal projections to $\H_0$, if necessary,  we may assume without loss of generality that
both vectors $\xi,\eta$ are cyclic vectors for $u$.
Approximating $\xi$ by elements in $\spn(u(G)\eta)$  and observing that $(s\mapsto \langle\xi, u_s\eta'\rangle)\in E$
for any $\eta'\in\spn(u(G)\eta)$ by translation invariance of $E$, it follows from any of the conditions
(1), (2), (3) that $s\mapsto \langle \xi, u_s\xi\rangle)\in E$, and a similar argument gives $(s\mapsto\langle\eta,u_s\eta\rangle)\in E$.
But then every summand in the polarisation identity
$$\lk \xi, u_s\eta\rk=\frac{1}{4} \sum_{k=0}^3i^k \lk \xi + i^k\eta, u_s(\xi+ i^k\eta)\rk$$
lies in $E$. This finishes the proof.
\end{proof}

\begin{example}\label{ex-C0} The above lemma applies to $E_0:=C_0(G)\cap B(G)$, which is closed in $B(G)$ under $\|\cdot\|_\infty$.
We do not know whether the conclusion of the above lemma applies to $D_p:=L^p(G)\cap B(G)$, so we do not know
whether $C_p^*(G)=C_{D_p,BG}^*(G)$ equals $C_{D_p,KLQ}^*(G)$. However, if we replace $D_p$ by
$E_p:=\spn(L^p(G)\cap P(G))$ we get $C_p^*(G)=C_{E_p,KLQ}^*(G)$.
\end{example}

\begin{convention}\label{ex con}
In what follows we shall often use the notation ``$C^*_E(G)$'' for
the {\em KLQ group algebra} attached to $E$. Recall that we have
$C_E^*(G)=C_{\overline{E}}^*(G)$ if $\overline{E}$ denotes the weak$^*$ closure of $E$ in $B(G)$.
We shall be careful to write
$C_{E, BG}^*(G)$ whenever we want to talk about the BG-group algebra attached to $E$
in cases where it does not obviously coincide
with $C_E^*(G)$ by any of the above results.
\end{convention}

\begin{remark}\label{co rem}
It will be relevant to us that algebra properties of $E$ are reflected in \emph{co}algebra properties of $C^*_E(G)$.  The basic result in this direction is due to Kaliszewski, Landstad and Quigg \cite[Corollary 3.13 and Proposition 3.16]{Kaliszewski-Landstad-Quigg:Exotic}.  To describe it, let $u_G:G\to \U(\M(C^*_\un(G)))$ be the universal representation of $G$ as in line \eqref{uni rep} above.
Define the \emph{comultiplication} homomorphism
$$
\delta:C^*_\un(G)\to \M(C^*_\un(G)\otimes C^*_\un(G))
$$
to be the integrated form of the diagonal representation $s\mapsto u_G(s)\otimes u_G(s)$, which exists by Remark \ref{uni prop}.  Consider the following diagram, where the horizontal arrows are induced by the canonical quotients:
$$
\xymatrix{ C^*_\un(G) \ar[d]^\delta \ar[r] & C^*_E(G) \ar@{=}[r] \ar@{-->}[d] & C^*_E(G) \ar@{-->}[d] \\
\M(C^*_\un(G)\otimes C^*_\un(G)) \ar[r] & \M(C^*_E(G)\otimes C^*_\un(G)) \ar[r] & \M(C^*_E(G)\otimes C^*_E(G)).}
$$
Then $E$ is a subalgebra of $B(G)$ if and only the rightmost dashed arrow can be filled in; and $E$ is an ideal in $B(G)$ if and only if the central dashed arrow can be filled in.
\end{remark}

We close this section with three theorems which
show that exotic group \cstar{}algebras allow new characterisations of classical notions from non-abelian harmonic analysis.  For discrete groups, these results can be found in \cite[Sections 2 and 3]{Brown-Guentner:New_completions}, although (1) if and only if (2) from Theorem \ref{amen} is much older, and due to Hulanicki \cite{Hulanicki}.  The general cases can be proved by slight elaborations of the arguments given there: see also \cite{Jolissaint} for the result on the Haagerup property in the general case.

\begin{theorem}\label{amen}
The following are equivalent:
\begin{enumerate}
\item $G$ is amenable;
\item $C_c(G)\cap B(G)$ is weak$^*$ dense in $B(G)$, i.e.\ $C^*_\un(G)=C^*_\red(G)$;
\item $E_p=\spn(L^p(G)\cap P(G))$ is weak$^*$ dense in $B(G)$ for some $p<\infty$, i.e.\ $C^*_\un(G)=C^*_p(G)$.
\end{enumerate}
\end{theorem}

\begin{theorem}\label{atmen}
The following are equivalent:
\begin{enumerate}
\item $G$ has the Haagerup approximation property;
\item $E_0:=B(G)\cap C_0(G)$ is weak$^*$ dense in $B(G)$, i.e.\ $C^*_\un(G)=C^*_{E_0}(G)$.
\end{enumerate}
\end{theorem}

The proof can be done as the proof of Corollary 3.4 in \cite{Brown-Guentner:New_completions}, where a similar result is
shown for discrete $G$ and $C_{C_0(G), BG}^*(G)$. By Example \ref{ex-C0} we know that this algebra coincides with
$C^*_{E_0}(G)$. The third theorem is about property (T):

\begin{theorem}\label{propt}
The following are equivalent:
\begin{enumerate}
\item $G$ has property (T);
\item If $E$ is a translation invariant ideal of $B(G)$ such that $E_0:=\spn(E\cap P(G))$ is weak$^*$ dense in $B(G)$,
then $E=B(G)$;
\item If $E$ is a translation invariant ideal of $B(G)$ such that
$C^*_\un(G)=C^*_{E,BG}(G)$, then $E=B(G)$.
\end{enumerate}
\end{theorem}
\begin{proof} The equivalence between (2) and (3) follows from the above discussions, since
$C_{E,BG}^*(G)=C_{E_0}^*(G)=C_{\overline{E_0}}^*(G)$ equals $C_\un^*(G)$ if and only if
$\overline{E_0}=B(G)$, where $\overline{E_0}$ denotes the weak$^*$ closure of $E_0$.

An analogue of the equivalence between (1) and (3) has been shown for discrete $G$ in
\cite[Proposition 3.6]{Brown-Guentner:New_completions}, but with $B(G)$ replaced by $\ell^\infty(G)$.
The proof of the general  case follows along similar lines:
Assume that $G$ has property (T) and let $E$ be a translation invariant ideal of $B(G)$
such that  $C_{E,BG}^*(G)=C_\un^*(G)$. Then there exists a faithful $E$-representation $u:G\to \U(\H)$,
e.g., take the direct sum of all GNS-representations attached to elements in $E\cap P(G)$. Then
$1_G$ is weakly contained in $u$, and therefore, by property (T), $1_G$ is a subrepresentation of $u$.
Hence there exists a unit vector $\xi\in \H$ such that $u_s\xi=\xi$ for all $s\in G$.
Since $u$ is an $E$-representation, there exists a sequence $(\xi_n)$ of unit vectors in $\H$
which converges to $\xi$ and such that $s\mapsto\phi_n(s)=\langle \xi_n,u_s\xi_n\rangle$ lies in
$E\cap P(G)$ for all $n\in \N$.
 It follows that $\phi_n\to 1_G$ in the norm topology
of $B(G)$. Since $B(G)$ is a Banach algebra, it follows that $E\cap P(G)$ contains an invertible
element of $B(G)$. Thus $E=B(G)$.

For the converse direction we can use the same arguments as
given in the  proof of
\cite[Proposition 3.6]{Brown-Guentner:New_completions}.
\end{proof}

\section{Exotic crossed products}\label{ex cp sec}

More details on the standard material on universal and reduced crossed products in this section can be found in \cite{Echterhoff:Crossed}, \cite[Appendix A]{Echterhoff-Kaliszewski-Quigg-Raeburn:Categorical}, and \cite{Williams:Crossed}.

Let $(A,\alpha)$ be a $G$-\cstar{}algebra, i.e.\ $A$ is a \cstar{}algebra equipped with a homomorphism $\alpha$ from $G$ to the \Star{}automorphisms of $A$ such that the map $s\mapsto \alpha_s(a)$ is (norm) continuous for all $a\in A$.  The natural class of representations of $(A,\alpha)$ are \emph{covariant pairs}: pairs $(\pi,u)$ consisting of a  \Star{}representation of $A$ and a (unitary) representation of $G$ on the same Hilbert space $\H$ that satisfy the relation
\begin{equation}\label{cov}
u_s^*\pi(a)u_s=\pi(\alpha_s(a))
\end{equation}
for all $a\in A$ and $s\in G$.

Let $C_c(G,A)$ denote the space of norm continuous, compactly supported functions from $G$ to $A$, equipped with the \Star{}algebra operations:
$$
f*g(t)\defeq\int_G f(s)\alpha_s(g(s^{-1}t))\dd{s}\quad\mbox{and}\quad f^*(t)\defeq \Delta(t)^{-1}\alpha_t(f(t^{-1}))^*.
$$
Note that any covariant pair $(\pi,u)$ integrates to a \Star{}representation $\pi\rtimes u$ of $C_c(G,A)$ via the formula
$$
\pi\rtimes u(f)=\int_G \pi(f(s))u_s \dd s.
$$
It will be useful for later purposes to note that this generalises to actions on Hilbert modules and multiplier algebras in a natural way.  Precisely, if $\mathcal{E}$ is a Hilbert $B$-module then a covariant pair $(\pi,u)$ for $(A,\alpha)$ on $\mathcal{E}$ consists of a \Star{}representation $\pi:A\to\Lb(\E)$ from $A$ to the adjointable operators on $\E$, and a (strongly continuous, unitary) representation $u:G\to\U(\mathcal{E})$ from $G$ to the group of unitary operators on $\E$ satisfying the compatibility relationship in line \eqref{cov} above.
Regarding a C*-algebra $D$ as a Hilbert module over itself, the isomorphism
$\M(D)\cong \mathcal L(D)$ shows that covariant representations into multiplier algebras with strongly continuous unitary part
are special cases of covariant representations on Hilbert modules.

Analogously to the case of group algebras, the following two completions of $C_c(G,A)$ have been very widely studied.

\begin{definition}\label{uni red cp}
The \emph{universal completion} $A\rtimes_{\alpha,\un} G$ is the completion of $C_c(G,A)$ with respect to the norm
$$
\|f\|_\un\defeq \sup\{\|\pi\rtimes u(f)\|:(\pi,u)\text{ a covariant pair}\}.
$$
Let $\pi$ be the representation of $A$ on the Hilbert $A$-module $L^2(G,A)$ defined by the formula $(\pi(a)\xi)(s)=\alpha_{s^{-1}}(a)\xi(s)$.  Let $\lambda\otimes1$ be the representation of $G$ on this Hilbert module defined by $((\lambda\otimes 1)_t\xi)(s)=\xi(t^{-1}s)$.  These representations define a covariant pair $(\pi,\lambda\otimes 1)$ on the Hilbert module $L^2(G,A)$.

The \emph{reduced crossed product} $A\rtimes_{\alpha,\red}G$ is\magc{,} by definition\magc{,} the completion of $C_c(G,A)$ for the norm
$$
\|f\|_\red\defeq \|\pi\rtimes (\lambda \otimes 1)(f)\|.
$$
\end{definition}

The universal crossed product has the universal property that any  covariant pair $(\pi,u)$ with values in a Hilbert module $\E$ (in particular, in a Hilbert space) for $(A,\alpha)$ integrates to a \Star{}homomorphism $$
\pi\rtimes_\un u: A\rtimes_{\alpha,\un}G\to \Lb(\E)
$$
to the bounded (adjointable) operators on the Hilbert space (module) and, conversely, that every nondegenerate \Star{}representation $\sigma$ of $A\rtimes_{\alpha,\un}G$ is the integrated
form of some (nondegenerate) covariant representation:
There is a universal covariant representation $(\iota_A,\iota_G)$ of $(A,G)$ into
$\M(A\rtimes_{\alpha,\un}G)$ such that  $\sigma=\pi\rtimes u$ with $\pi=\sigma\circ \iota_A$ and $u=\sigma\circ \iota_G$. On the level of $\contc(G,A)$ the universal covariant pair is given by the formulas
\begin{equation}\label{eq:univ-cov-rep}
\iota_A(a)f(s)=af(s),\quad \iota_G(t)f(s)=\alpha_t(f(t^{-1}s)).
\end{equation}

Analogously to the group case, an exotic crossed product is roughly a completion of $C_c(G,A)$ for a norm between the maximal and reduced norms.  Motivated by both examples and applications, it seems reasonable to ask for some compatibility between such exotic crossed products as $A$ varies.  The minimal reasonable requirement here seems to be compatibility with \Star{}homomorphisms (we will discuss some stronger requirements later): to make this precise, note that if $\phi:A\to B$ is a $G$-equivariant \Star{}homomorphism then the function
\begin{equation}\label{*hom}
\phi\rtimes_c G:C_c(G,A)\to C_c(G,B),\quad f\mapsto \phi\circ f
\end{equation}
is a \Star{}homomorphism, and moreover, the assignment $\phi\mapsto \phi\rtimes_c G$ is functorial.

\begin{definition}\label{ex cp}
A \emph{crossed product} is an assignment to each $G$-\cstar{}algebra $(A,\alpha)$ of a completion $A\rtimes_{\alpha,\mu}G$ of $C_c(G,A)$ for a norm $\|\cdot \|_\mu$ such that:
\begin{enumerate}
\item for all $f\in C_c(G,A)$,
$$
\|f\|_\un \geq \|f\|_\mu\geq \|f\|_\red;
$$
\item for any equivariant \Star{}homomorphism $\phi:(A,\alpha)\to (B,\beta)$ the \Star{}homomorphism $\phi\rtimes_{c} G$ of line \eqref{*hom} extends to a \Star{}homomorphism $\phi\rtimes_{\mu} G:A\rtimes_{\alpha,\mu}G\to B\rtimes_{\beta,\mu}G$.
\end{enumerate}

A crossed product is \emph{exotic} if the associated norm differs from the maximal and reduced norms (on at least one $G$-\cstar{}algebra each).
\end{definition}

Thus a crossed product is a functor from the category of $G$-\cstar{}algebras and equivariant \Star{}homomorphisms to the category of \cstar{}algebras and \Star{}homomorphisms that sits between the universal and reduced completions in some sense.  One might wonder why we require that we only consider completions of $C_c(G,A)$ by norms $\|\cdot\|_\mu$ which dominate the reduced norm $\|\cdot\|_\red$. We come back to this point in Subsection \ref{rem-exotic-norms}

There are natural extensions of both the Brown-Guentner and Kaliszewski-Landstad-Quigg exotic group algebra constructions to exotic crossed products.  Here is the Brown-Guentner construction.
\begin{definition}\label{bg cp}
Let $C^*_E(G)$ be a group \cstar{}algebra as in Convention \ref{ex con}, and $(A,\alpha)$ a dynamical system.  The \emph{Brown-Guentner crossed product} $A\rtimes_{\alpha,E_{BG}}G$ (for short: BG crossed product) is the completion of $C_c(G,A)$ for the norm
$$
\|f\|_{E_{BG}}\defeq\sup\{\|\pi\rtimes u(f)\| ~|~(\pi,u)\text{ a covariant pair such that $u$ extends to }C^*_E(G)\}.
$$
\end{definition}

In order to define the Kaliszewski-Landstad-Quigg construction of crossed products, we need a little more notation.   Let $(\iota_A,\iota_G)$ be the universal covariant representation of $(A,G)$ into $\M(A\rtimes_{\alpha,\un}G)$ as in line \eqref{eq:univ-cov-rep} above.  Let $u_G\colon G\to \M(C^*_\un(G))$ be the universal representation of $G$ as in line \eqref{uni rep} above (that is, $u_G$ coincides with $\iota_G:G\to\U\M(\C\rtimes_\un G)$
if we identify $C_\un^*(G)$ with $\C\rtimes_\un G$).  Then the maps
$$
\iota_A\otimes 1:A\to \M(A\rtimes_{\alpha,\un}G\otimes C^*_\un(G)),\quad \iota_G\otimes u_G:G\to \U \M(A\rtimes_{\alpha,\un}G\otimes C^*_\un(G))
$$
define a covariant pair for $(A,\alpha)$ (here and throughout the rest of the paper, ``$\otimes$'' denotes the spatial tensor product of \cstar{}algebras).  The integrated form of this covariant pair (which exists by the universal property of $A\rtimes_{\alpha,\un}G$) is a \Star{}homomorphism
\begin{equation}\label{dual coact}
\widehat{\alpha}:A\rtimes_{\alpha,\un}G \to\M(A\rtimes_{\alpha,\un}G \otimes C^*_\un(G))
\end{equation}
called the \emph{dual coaction} associated to $(A,\alpha)$.

\begin{definition}\label{klq cp}
Let $C^*_E(G)$ be a group \cstar{}algebra, and $(A,\alpha)$ a dynamical system.  Let $q_E:C^*_\un(G)\to C^*_E(G)$ denote the canonical quotient map, and let
$$
\text{id}\otimes q_E:\M(A\rtimes_{\alpha,\un}G \otimes C^*_\un(G))\to \M(A\rtimes_{\alpha,\un}G \otimes C^*_E(G))
$$
denote the extension of the canonical tensor product \Star{}homomorphism to the multiplier algebras.

The \emph{Kaliszewski-Landstad-Quigg crossed product} $A\rtimes_{\alpha,E_{KLQ}}G$ is the completion of $C_c(G,A)$ for the norm
$$
\|f\|_{E_{KLQ}}\defeq \|(\text{id}\otimes q_E)\circ \widehat{\alpha}(f)\|.
$$
\end{definition}

The properties of the BG and KLQ crossed products that we will use are recorded below.  We will discuss some more properties of these functors in the next section.  Proofs of these results (in a slightly different form) and other basic facts about BG and KLQ crossed products can be found in \cite[Appendix A]{Baum-Guentner-Willett}, \cite[Section~6]{Kaliszewski-Landstad-Quigg:Exotic}, and \cite[Section 5]{Buss-Echterhoff:Exotic_GFPA}.

\begin{proposition}\label{basic prop}
Let $C^*_E(G)$ be an exotic group algebra.   The following facts hold for the associated BG and KLQ crossed products.
\begin{enumerate}
\item The BG and KLQ crossed products are both functorial for equivariant \Star{}homomorphisms, and in particular are crossed products in the sense of Definition \ref{ex cp}.
\item The correspondences $E\mapsto \rtimes_{E_{BG}}$ and $\rtimes_{E_{BG}}\mapsto (\C\rtimes_{E_{BG}}G)'$ are mutually inverse bijections between the collection of all $BG$ functors and all weak$^*$-closed translation invariant subspaces of $B(G)$ that contain $B_\red(G)$.  In particular, $\C\rtimes_{E_{BG}}G$ identifies with $C^*_E(G)$ via a \Star{}isomorphism that extends the identity map on $C_c(G)$.
\item \label{klq triv} The correspondences $E\mapsto \rtimes_{E_{KLQ}}$ and $\rtimes_{E_{BG}}\mapsto (C\rtimes_{E_{KLQ}}G)'$ are mutually inverse bijections between the collection of all $KLQ$ functors and all weak$^*$-closed translation invariant ideals in $B(G)$ that contain $B_\red(G)$.  In particular, for any group \cstar{}algebra $C^*_E(G)$, $\C\rtimes_{E_{KLQ}}G$ identifies with $C^*_{\langle E\rangle}(G)$ via a \Star{}isomorphism that extends the identity map on $C_c(G)$, where $\langle E\rangle$ is the weak$^*$-closed ideal in $B(G)$ generated by $E$.
\end{enumerate}
\end{proposition}

In the next section, we will study functorial properties of the BG and KLQ crossed products in much more detail.

We conclude this section with some rather unnatural examples that are useful for constructing crossed products with `bad' properties.  For yet another construction of exotic crossed products, see \cite[Section 2.4 and Corollary 4.20]{Buss-Echterhoff-Willett}.

\begin{example}\label{ce cp}
Let $\mathcal{S}$ be a collection of $G$-\cstar{}algebras.  For any $G$-\cstar{}algebra $(A,\alpha)$, define a seminorm on $C_c(G,A)$ by
$$
\|f\|_{\mathcal{S},0}\defeq \sup \{\|\phi\rtimes_{c} G(f)\|_\un:\phi\in \Mor_G(A,B) \text{ for some } B\in \mathcal{S}\},
$$
where $\Mor_G(A,B)$ denotes the set of $G$-equivariant \Star{}homomorphisms $A\to B$. We then define a norm on $C_c(G,A)$ by
$$
\|f\|_{\mathcal{S}}\defeq \max\{\|f\|_\red,\|f\|_{\mathcal{S},0}\}.
$$
Let $A\rtimes_{\alpha,\mathcal{S}}G$ be the associated completion.  Then the assignment $(A,\alpha)\mapsto A\rtimes_{\alpha,\mathcal{S}}G$ is a crossed product functor (see \cite{Buss-Echterhoff-Willett}*{Lemma~2.5}).
\end{example}

\section{Properties of crossed products}\label{cp props sec}

We start this section by discussing some strong functoriality properties that a crossed product functor can have and give some useful characterisations of these.  We then discuss some applications to $K$-theory computations, duality theory and tensor products. We also discuss in Section~\ref{rem-exotic-norms} ``pseudo crossed products'', which are certain quotients of the full crossed product that do not necessarily lie above the reduced crossed product.

Before we start stating the properties of interest, we need
to give a brief discussion about crossed products of $G$-equivariant Hilbert modules and correspondences.
If $(B,\beta)$ is a $G$-algebra and $\E$ is a Hilbert $B$-module, then a compatible action of $G$ on $\E$ is
a strongly continuous homomorphism $\gamma:G\to \Aut(\E)$ such that
$$\gamma_s(x b)=\gamma_s(x)\beta_s(b)\quad\text{and}\quad \langle\gamma_s(x),\gamma_s(y)\rangle=
\beta_s(\langle x,y\rangle),$$
for all $x,y\in E, a\in A$ and $s\in G$. If $(A,\alpha)\mapsto A\rtimes_{\alpha,\mu}G$ is an exotic crossed-product functor
we may extend this functor to equivariant Hilbert modules as follows: If $(\E,\gamma)$ is a $G$-equivariant $(B,\beta)$
Hilbert module as above, then
there is a well-known canonical $C_c(G,B)$-valued inner product on $C_c(G,\E)$ together with a
compatible right module action of $C_c(G,B)$ on $C_c(G,\E)$ given by
$$
\braket{x}{y}_{C_c(G,B)}(t)=\int_G\beta_{s^{-1}}(\braket{x(s)}{y(st)}_B\,ds, \quad
(x\cdot \varphi)(t)=\int_G x(s)\beta_{s^{-1}}(\varphi(s^{-1}t))\,ds
$$
for $x,y\in C_c(G,\E)$ and $\varphi\in C_c(G,B)$. If we regard $C_c(G,B)$ a subalgebra of $B\rtimes_{\beta,\mu}G$, we obtain
a norm $\| x\|_\mu=\sqrt{\|\braket{x}{x}\|_\mu}$ on $C_c(G,\E)$. The actions and inner products then extend to the
$\mu$-completions and we obtain a $B\rtimes_{\beta,\mu} G$-Hilbert module $\E\rtimes_{\gamma,\mu}G$.

If $(A,\alpha)$ is another $G$-algebra, then a {\em $G$-equivariant $(A,\alpha)-(B,\beta)$ correspondenc}e is
a triple $(\E,\gamma,\phi)$ in which $(\E,\gamma)$ is a $G$-equivariant $(B,\beta)$-Hilbert module and
$\phi:A\to \mathcal L(\E)$ is an $\alpha-\Ad\gamma$-equivariant \Star{}homomorphism (possibly degenerate).
There is a category $\Cor(G)$ in which the objects are $G$-C*-algebras and the morphisms are equivalence classes of $(A,\alpha)-(B,\beta)$ correspondences, where $(\E,\gamma,\phi)\sim (\E',\gamma',\phi')$ if there is an isomorphism $\phi(A)\E\congto \phi'(A)\E'$ of Hilbert $B$-module commuting the left actions of $A$.
Composition of correspondences is given by taking internal tensor products
$$[(\E,\gamma,\phi)][(\F,\nu,\psi)]=[(\E\otimes_B\F, \gamma\otimes\nu,\phi\otimes1)].$$
We  write $\Cor:=\Cor(\{e\})$ for the correspondence category of the trivial group $\{e\}$.
Isomorphisms in the correspondence categories are precisely the Morita equivalences, i.e.\ correspondences
where $\phi:A\to \mathcal L(\E)$ induces an isomorphism $A\cong \K(\E)$.
Correspondence categories have been studied extensively in the literature (e.g. see
\cite{Echterhoff-Kaliszewski-Quigg-Raeburn:Categorical}), where usually the homomorphisms
$\Phi: A\to \mathcal L(\E)$ are assumed to be nondegenerate. But notice that every correspondence in our sense
is equivalent  to a nondegenerate correspondence, so the resulting categories are equivalent.

\begin{definition}\label{cp props}
Let $(A,\alpha)\to A\rtimes_{\alpha,\mu}G$ be a crossed-product functor. This functor:
\begin{enumerate}
\item  {\em extends to generalised homomorphisms}  if for any (possibly degenerate) $G$\nb-equivariant \Star{}homomorphism
 $\phi:A\to \M(B)$  there exists a \Star{}homomor\-phism $\phi\rtimes_{\mu} G:A\rtimes_{\mu}G\to \M(B\rtimes_{\mu}G)$
 which is given on the level of functions $f\in C_c(G,A)$ by $f\mapsto \phi\circ f$ in the
 sense that $\phi\rtimes_\mu G (f)g=(\phi\circ f)*g$ for all $g\in C_c(G,B)$;
 \item  has the {\em ideal property} if for every $G$-invariant closed ideal in a $G$-algebra $A$, the inclusion map $\iota:I\into A$ descends
to an injective *-homomorphism $\iota\rtimes_{\mu}G: I\rtimes_{\mu}G\into A\rtimes_{\mu}G$;
\item is {\em strongly Morita compatible} if for every $G$\nb-equivariant $(A,\alpha)-(B,\beta)$
equivalence bimodule $(\E,\gamma)$, the left action of $C_c(G, A)$ on $C_c(G,\E)$ given by
$$
(\phi\rtimes G(f)x)(t)\defeq\int_G \phi(f(s))\gamma_s(x(s^{-1}t))\,ds
$$
extends to an action of
$A\rtimes_{\alpha,\mu}G$ on $\E\rtimes_{\gamma,\mu}G$ such that $\E\rtimes_{\gamma,\mu}G$ becomes
an $A\rtimes_{\alpha,\mu}G-B\rtimes_{\beta,\mu}G$ equivalence bimodule;
\item is a \emph{correspondence functor} if for every $G$\nb-equivariant $(A,\alpha)-(B,\beta)$
correspondence $(\E,\gamma,\phi)$, the left action of $C_c(G, A)$ on $C_c(G,\E)$ above extends to an action $\phi\rtimes_\mu G$ of
$A\rtimes_{\alpha,\mu}G$ on $\E\rtimes_{\gamma,\mu}G$ such that $(\E\rtimes_{\gamma,\mu}G,\phi\rtimes_\mu G)$ becomes
an $A\rtimes_{\alpha,\mu}G-B\rtimes_{\beta,\mu}G$ correspondence;
\item has the {\em (full) projection property} if for every $G$\nb-algebra $A$ and every
$G$\nb-invariant (full) projection $p\in \M(A)$, the inclusion $\iota: pAp\into A$ descends to
a  faithful homomorphism $\iota\rtimes_{\mu}G: pAp\rtimes_{\alpha,\mu}G\to A\rtimes_{\alpha,\mu}G$;
\item has the {\em (full) hereditary-subalgebra property} if for every (full) hereditary
$G$\nb-invariant subalgebra $B$ of $A$, the inclusion $\iota:B\into A$ descends to a  faithful map
$\iota\rtimes_{\mu}G: B\rtimes_{\alpha,\mu}G\to A\rtimes_{\alpha,\mu}G$;
\item has the \emph{cp map property} if for any completely positive and $G$-equivariant map $\phi:A\to B$ of $G$-algebras, the map
$$
C_c(G,A)\to C_c(G,B),\quad f\mapsto \phi\circ f
$$
extends to a completely positive map from $A\rtimes_{\alpha,\mu} G$ to $B\rtimes_{\beta,\mu} G$.
\end{enumerate}
\end{definition}

\begin{remark}
Of course, it follows from the requirements for a correspondence functor $\rtimes_\mu$ that it extends
to a functor  $\rtimes_\mu: \Cor(G)\to \Cor$, and similarly a functor with the cp map property extends to a functor from the category of $G$-\cstar{}algebras and equivariant completely positive maps to the category of \cstar{}algebras and completely positive maps.
\end{remark}

We now record the relationships between these various properties.  Proposition~\ref{id props} is proved in \cite[Section~3]{Buss-Echterhoff-Willett}, and Theorems \ref{mor props} and \ref{corr props} in \cite[Section~4]{Buss-Echterhoff-Willett}.

\begin{proposition}\label{id props}
The following are equivalent for a crossed product functor:
\begin{enumerate}
\item the ideal property;
\item extension to generalised morphisms.
\end{enumerate}
\end{proposition}

\begin{theorem}\label{mor props}
The following are equivalent for a crossed product functor:
\begin{enumerate}
\item strong Morita compatibility;
\item the full hereditary subalgebra property;
\item the full projection property.
\end{enumerate}
\end{theorem}

\begin{theorem}\label{corr props}
The following are equivalent for a crossed product functor:
\begin{enumerate}
\item being a correspondence functor;
\item the projection property;
\item the hereditary subalgebra property;
\item the cp map property;
\item having the properties in Theorem \ref{mor props} and Proposition \ref{id props}.
\end{enumerate}
\end{theorem}

The properties of BG and KLQ functors listed below are proved in \cite[Sections 4 and 5]{Buss-Echterhoff-Willett}.

\begin{example}\label{bg props}
The BG crossed product associated to a group algebra $C^*_E(G)$ always has the ideal property.  It is strongly Morita compatible if and only if it is a correspondence functor, if and only if the canonical quotient map $C^*_\un(G)\onto C^*_E(G)$ is an isomorphism.
\end{example}

\begin{example}\label{klq props}
KLQ crossed products are always correspondence functors, and thus have all the properties considered above.  It is conceivable that all correspondence functors are KLQ functors.  This seems unlikely, however, partly as there is a construction of correspondence functors that are not obviously KLQ functors: see \cite[Section 2.4 and Corollary 4.20]{Buss-Echterhoff-Willett}.
\end{example}

On the other hand, the ideal property does not always hold, as the following example shows.

\begin{example}\label{bad ex}
Let $\mathcal{S}$ consist of $C_0(0,1]$ equipped with the trivial action, and apply the construction of Example \ref{ce cp} for this $\mathcal{S}$.  Then the inclusion $C_0(0,1]\to C[0,1]$ of trivial $G$-algebras does not induce an injection on crossed products for any non-amenable $G$.

We do not, however, know if there are strongly Morita compatible crossed products without the ideal property (we guess the answer is yes, by an elaboration of the above, but the details are currently elusive).
\end{example}

The following theorem is one of the main applications \cite[Section 6]{Buss-Echterhoff-Willett} of our correspondence functor machinery; it can be regarded as another good functoriality property of correspondence functors.

\begin{theorem}\label{desc amen}
Say $(A,\alpha)\mapsto A\rtimes_{\alpha,\mu}G$ is a correspondence functor.  Then there exists a \emph{descent functor} $\rtimes_\mu:KK^G\to KK$ that agrees with $\rtimes_\mu$ on objects, and on morphisms coming from equivariant \Star{}homomorphisms.

Moreover, if $G$ is $K$-amenable, then the canonical quotients
$$
A\rtimes_{\alpha,\un}G\onto A\rtimes_{\alpha,\mu}G \onto A\rtimes_{\alpha,\red}G
$$
are $KK$-equivalences.
\end{theorem}

\subsection{K-theory of exotic group algebras}

\begin{corollary}\label{ideal}
Let $G$ be a $K$-amenable group, and let $E$ be a translation invariant
 algebraic ideal in $B(G)$ which contains  $B_c(G)=B(G)\cap\contc(G)$.
Then the canonical quotient maps
$$
C^*_{\un}(G)\onto C^*_{E,BG}(G)\onto C^*_\red(G),\quad C^*_{\un}(G)\onto C^*_{E,KLQ}(G)\onto C^*_\red(G)
$$
are all $KK$-equivalences.
\end{corollary}
\begin{proof}
Using Theorem \ref{desc amen}, if suffices to show that $C^*_{E,BG}(G)$ and $C^*_{E,KLQ}(G)$ are of the form $\C\rtimes_\mu G$ for some correspondence functor $\rtimes_\mu$.

For $C^*_{E,KLQ}(G)$, let $\overline{E}$ be the weak$^*$ closure of $E$ in $B(G)$, which is an ideal by weak$^*$ continuity of multiplication on $B(G)$.   Let $\rtimes_\mu$ be the KLQ crossed product functor associated to $C^*_{\overline{E},KLQ}(G)$.  Proposition \ref{basic prop} part \eqref{klq triv} implies that $\C\rtimes_\mu G$ identifies with $C^*_{\overline{E},KLQ}(G)$; however $C^*_{\overline{E},KLQ}(G)$ is clearly the same as $C^*_{E,KLQ}(G)$ by definition of KLQ group algebras.

For $C^*_{E,BG}(G)$, let $\widetilde{E}$ be the span of $P(G)\cap E$, and let $F$ be the weak$^*$-closure of $\widetilde{E}$.
Then  $C^*_{E,BG}(G)=C_{F, KLQ}^*(G)$ by Corollary \ref{same same}. As $P(G)$ is closed under products in $B(G)$
and every element in $B(G)$ is a linear combination of elements in $P(G)$, $\widetilde{E}$ is an ideal in $B(G)$ and $F$ is a weak$^*$-closed ideal.  Thus the result follows from the result for
KLQ group algebras.
\end{proof}

\begin{example}
Say $G=F_2$ or $G=SL(2,\mathbb{R})$.  Then $G$ is $K$-amenable, so Corollary \ref{ideal} implies that the uncountably many exotic group \cstar{}algebras $C^*_p(G)$ from Theorem \ref{lp diff} are all $KK$-equivalent.
\end{example}

\begin{example}\label{no go}
It is tempting from the above to guess that if $G$ is $K$-amenable, then \emph{all} exotic group algebras (or even crossed products) have the same $K$-theory.  This is false: in fact any non-amenable group admits an exotic group algebra such that the canonical quotient $C^*_E(G)\to C^*_\red(G)$ does not even induce an isomorphism on $K$-theory.  This can be achieved by setting $E=B_\red(G)\oplus\C1$, for example.

Compare also Remark \ref{ideal nec} in this regard, which implies that Corollary \ref{ideal} is in some sense the best possible result that can be deduced about $K$-theory of group algebras using our correspondence functor machinery.
\end{example}

\subsection{Duality}

\begin{definition}\label{dual fun}
A crossed product functor $\rtimes_\mu$ is a \emph{duality functor} if there is a \Star{}homomorphism $\widehat{\alpha}_\mu$ making the diagram below commute
$$
\xymatrix{A\rtimes_{\alpha,\un}G \ar[d]  \ar[r]^-{\widehat{\alpha}} & \M(A\rtimes_{\alpha,\un}G \otimes C^*_\un(G)) \ar[d] \\
A\rtimes_{\alpha,\mu}G  \ar@{-->}[r]^-{\widehat{\alpha}_\mu} & \M(A\rtimes_{\alpha,\mu}G \otimes C^*_\un(G))},
$$
where $\widehat{\alpha}$ is the dual coaction of line \eqref{dual coact} above, and the vertical maps are the canonical quotients.
\end{definition}

For the proof of the following theorem see \cite[Section~6]{Buss-Echterhoff-Willett}.

\begin{theorem}\label{dual the}
Correspondence functors are duality functors.
\end{theorem}

\begin{remark}
Let $C^*_E(G)$ be a group algebra as in Convention \ref{ex con}.  It is not difficult to see that the associated BG crossed product is a duality functor if and only if $E$ is an ideal in $B(G)$.  In particular, it follows from Example \ref{bg props} that Theorem \ref{dual the} is not optimal.
\end{remark}

\begin{remark} Note that every duality functor $\rtimes_\mu$ admits a version of Imai-Takai duality:
The homomorphism $\widehat\alpha_\mu$ is a coaction and there is a canonical isomorphism $A\rtimes_{\alpha,\mu}G\rtimes_{\dual\alpha_\mu}\dualG\cong A\otimes\K(L^2(G))$.
We refer to \cite[Section~6]{Buss-Echterhoff-Willett} for more details on this and to
\cite{Buss-Echterhoff:maximal-dual, Kaliszewski-Landstad-Quigg:coaction-functors} for results which show
how duality techniques combined with Theorem \ref{dual the}
can be efficiently used to extend correspondence crossed-product functors to other categories like Fell-bundles over $G$ or
partial $G$-actions.
\end{remark}

\begin{remark}\label{ideal nec}
If $\rtimes_\mu$ is a duality functor and $E$ is the dual space of $\C\rtimes_\mu G$, thought of as a subspace of $B(G)$, then $\C\rtimes_\mu G\cong C^*_E(G)$ carries a coaction of $G$ and hence $E$ is an ideal of $B(G)$ by Remark~\ref{co rem}.  Combined with Proposition \ref{basic prop} part \eqref{klq triv} and Example \ref{klq props}, it follows that an exotic group algebra $C^*_E(G)$ is of the form $\C\rtimes_\mu G$ for a correspondence functor $\rtimes_\mu$ if and only if the dual space of $C^*_E(G)$ is an ideal in $B(G)$.

In particular, the result of Corollary \ref{ideal} is in some sense the optimal application of Theorem \ref{desc amen} to computing the $K$-theory of exotic group algebras.
\end{remark}

\subsection{Tensor products}

As an example of an application of Theorem \ref{corr props} that does not appear in our paper \cite{Buss-Echterhoff-Willett}, here we discuss the relationship of crossed products and spatial tensor products.  In the next section, we will apply this to discuss the relationship between general crossed products and the reduced crossed product.

Let $\rtimes_\mu$ be a crossed product for $G$, which is functorial for generalised morphisms.  Let $(A,\alpha)$ be a $G$-\cstar{}algebra, and let $(B,\text{id})$ be a trivial $G$-\cstar{}algebra.  Then there is an equivariant \Star{}homomorphism $A\to \M(A\otimes B)$ defined by $a\mapsto a\otimes 1$, and functoriality for generalised morphisms implies that this integrates to a \Star{}homomorphism
\begin{equation}\label{into mult}
A\rtimes_{\alpha,\mu} G \to \M((A\otimes B)\rtimes_{\alpha\otimes \text{id},\mu} G).
\end{equation}
Assume now moreover that either $B$ or $A\rtimes_{\alpha,\mu} G$ is nuclear.  Then as the natural \Star{}homomorphism $B\to \M((A\otimes B)\rtimes_{\alpha\otimes \text{id},\mu} G)$ commutes with the image of the \Star{}homomorphism in line \eqref{into mult}, our nuclearity assumptions give a \Star{}homomorphism
$$
(A\rtimes_{\alpha,\mu} G) \otimes B \to \M((A\otimes B)\rtimes_{\alpha\otimes \text{id},\mu} G);
$$
checking on generators, it is not difficult to see that the image of this \Star{}homomorphism is in fact in $(A\otimes B)\rtimes_{\alpha\otimes \text{id},\mu} G$, and thus we have a canonical \Star{}homomorphism
\begin{equation}\label{tp map}
(A\rtimes_{\alpha,\mu} G) \otimes B \to (A\otimes B)\rtimes_{\alpha\otimes \text{id},\mu} G.
\end{equation}

\begin{definition}\label{tp prop}
Let $(A,\alpha)\mapsto A\rtimes_{\alpha,\mu} G$ be a crossed product functor which is functorial for generalised morphisms.  The functor has the \emph{tensor product property} if whenever $B$ is a trivial $G$-\cstar{}algebra and $(A,\alpha)$ is a $G$-\cstar{}algebra such that one of $B$ or $A\rtimes_{\alpha,\mu} G$ is nuclear, the map in line \eqref{tp map} above is an isomorphism.
\end{definition}

Since the map in line~\eqref{tp map}  is always surjective the tensor product property is equivalent to the injectivity of~\eqref{tp map}.

\begin{example}\label{bg tp}
The BG functor associated to a group algebra $C^*_E(G)$ has the tensor product property.  Indeed, we have already noted that BG functors are functorial for generalised morphisms.  It suffices to prove that there is is a faithful representation $\pi$ of  $(A\rtimes_{\alpha,E_{BG}} G) \otimes B$ that extends to $(A\otimes B)\rtimes_{\alpha\otimes \text{id},E_{BG}} G$.  Now, by definition of the BG crossed product and the spatial tensor product, we may take a faithful $\pi$ of the form $(\sigma\rtimes u)\otimes \rho$ for some covariant pair $(\sigma,u)$ for $(A,\alpha)$ with $u$ an $E$-representation, and $\rho$ a representation of $B$.  The desired extension is then the integrated form of $(\sigma\otimes \rho,u\otimes \text{id})$, which is a covariant pair for $A\otimes B$ with $u\otimes \text{id}$ an $E$-representation.
\end{example}

\begin{remark}\label{no tp case}
If $B$ is unital, then the discussion above makes sense without assuming that $\rtimes_\mu$ is functorial for generalised morphisms, and thus the map in line \eqref{tp map} makes sense for any unital nuclear $B$.  At this level of generality, the map in line \eqref{tp map} can certainly fail to be an isomorphism, even for $B=M_2(\C)$: use the construction in Example \ref{ce cp} with $\mathcal{S}$ consisting of $M_2(\C)$ equipped with the trivial representation \cite[Example 2.6]{Buss-Echterhoff-Willett}.
\end{remark}

\begin{remark}\label{no tp}
On the other hand, the map in line \eqref{tp map} is always an isomorphism for unital and commutative $B$, and always an isomorphism for $B$ commutative if $\rtimes_\mu$ is functorial for generalised homomorphisms \cite[Remark 3.7 and Lemma 3.6]{Buss-Echterhoff-Willett} (it is false for general commutative \cstar{}algebras without assuming the ideal property: compare Example \ref{bad ex} above).
\end{remark}

\begin{proposition}\label{nuc out}
Correspondence functors have the tensor product property.
\end{proposition}

\begin{proof}
We must show that the map in line \eqref{tp map} is injective, so assume for contradiction that $x$ is an element of the kernel of the map above.  Then there is a state $\phi$ on $B$ such that the slice map
$$
1_{A\rtimes_{\alpha,\mu}G}\otimes \phi:(A\rtimes_{\alpha,\mu}G)\otimes B\to A\rtimes_{\alpha,\mu}G
$$
is non-zero on $x$.  As the slice map $1_A\otimes \phi:A\otimes B\to B$ is an equivariant completely positive map, the cp map property implies that it integrates to a map $(1_A\otimes \phi)\rtimes_\mu G$ on the $\mu$-crossed products.  This gives rise to a  diagram
$$
\xymatrix{(A\rtimes_{\alpha,\mu}G)\otimes B \ar[r] \ar[d]^-{1_{A\rtimes_{\alpha,\mu}G}\otimes \phi} & (A\otimes B)\rtimes_{\alpha\otimes \text{id},\mu}G \ar[d]^{(1_A\otimes \phi)\rtimes_\mu G}\\
A\rtimes_{\alpha,\mu}G \ar@{=}[r] & A\rtimes_{\alpha,\mu}G },
$$
which commutes by checking on the dense subalgebra $C_c(G,A)\otimes_{\operatorname{alg}} B$.  As the `right-down' composition sends $x$ to zero, and the `down-right' composition does not, we have our contradiction.
\end{proof}

\subsection{Crossed products and the reduced group \cstar{}algebra}\label{rem-exotic-norms}

Throughout this paper, we only consider exotic group algebras and crossed products that dominate the reduced group algebra.  Here we discuss the sort of degeneracy that can occur if one does not do this; some of the ideas underlying this section were pointed out to us by Joachim Cuntz.

For the purposes of this subsection only, by a \emph{pseudo crossed product} functor we mean a functor that satisfies all of the conditions of Definition \ref{ex cp} except possibly that it does not dominate the reduced crossed product norm, and that possibly the norm $\|\cdot\|_\mu$ is only a semi-norm on $C_c(G,A)$ for some \cstar{}algebras $(A,\alpha)$.  Similarly, a \emph{pseudo group algebra} is a \cstar{}algebra completion of $C_c(G)$ that satisfies the conditions of Definition \ref{ex gp alg} except that possibly the norm $\|\cdot\|_\mu$ does not dominate the reduced norm, and is only a semi-norm on $C_c(G)$.   Note that the definitions of the BG and KLQ crossed products still make sense if we allow pseudo group \cstar{}algebras as the input, but then give rise to pseudo crossed products.  Note moreover that the definitions of functoriality for generalised morphisms, and of the tensor product property still make sense.

\begin{lemma}\label{degen}
Say $\rtimes_\mu$ is a pseudo-crossed product functor that is functorial for generalised morphisms, and such that the crossed product $C_0(G)\rtimes_{\tau,\mu} G$ (where $\tau$ is the left translation action) is non-zero.  Then $\C\rtimes_\mu G$ dominates the reduced group \cstar{}algebra.

If moreover $\rtimes_\mu$ has the tensor product property, then $\rtimes_\mu$ dominates $\rtimes_\red$.
\end{lemma}

\begin{proof}
From generalised functoriality, the unit inclusion $\C\to\M(C_0(G))$ induces a \Star{}homomorphism $\C\rtimes_\mu G\to \M(C_0(G)\rtimes_{\tau,\mu} G)$.  If $C_0(G)\rtimes_\mu G$ is non-zero, then it canonically identifies with $\K(L^2(G))$ (the compact operators on $L^2(G)$), as this is true for $\rtimes_\un$, and as $\K(L^2(G))$ is simple.  Thus $\M(C_0(G)\rtimes_{\tau,\mu} G)$ identifies canonically with $\Lb(L^2(G))$.   However, it is easy to see that the composition
$$
\C\rtimes_\mu G\to \M(C_0(G)\rtimes_{\tau,\mu} G)\cong \Lb(L^2(G))
$$
is the integrated form of the regular representation.

Assume now that  in addition $\rtimes_\mu$ has the tensor product property. Then, by generalised functoriality,  for any $G$-\cstar{}algebra the canonical inclusion $A\to \M(A\otimes C_0(G))$, $a\mapsto a\otimes 1$ gives rise to a \Star{}homomorphism
$$
A\rtimes_{\alpha,\mu}G \to \M((A\otimes C_0(G))\rtimes_{\alpha\otimes \tau,\mu}G).
$$
On the other hand, using Fell's trick, the tensor product property and the fact (as above) that $C_0(G)\rtimes_{\tau,\mu}G=\K(L^2(G))$ gives
\begin{align*}
\M((A\otimes C_0(G))\rtimes_{\alpha\otimes \tau,\mu}G) & \cong \M((A\otimes C_0(G))\rtimes_{\text{id}\otimes \tau,\mu}G) \\ & \cong \M(A\otimes (C_0(G)\rtimes_{\tau,\mu}G)) \\ & \cong \M(A\otimes \K(L^2(G))).
\end{align*}
It is not difficult to check that the composed map
$$
A\rtimes_{\alpha,\mu}G\to \M(A\otimes \K(L^2(G)))
$$
is induced by an integrated form of the regular representation.
\end{proof}

A similar argument can be used for functors without the ideal property whenever $G$ admits a unital \cstar{}dynamical system $(A,\alpha)$ such that the maximal crossed product $A\rtimes_{\alpha,\un}G$ is simple.  Indeed, say then $\rtimes_\mu$ is a pseudo-crossed product.  If $A\rtimes_{\mu,\alpha}G$ is non-zero, then the canonical quotient map $A\rtimes_{\alpha,\un}G\onto A\rtimes_{\alpha,\mu}G$ is an isomorphism (and similarly for the reduced crossed product); hence the unit inclusion $\C\to A$ induces a commutative diagram
$$
\xymatrix{ \C\rtimes_\red G \ar@{^(->}[d] & \C\rtimes_\un G\ar[l]\ar[d] \ar[r] & \C\rtimes_\mu G \ar[d]\\
A\rtimes_{\alpha,\red} G & A\rtimes_{\alpha,\un} G\ar@{=}[l] \ar@{=}[r] & A\rtimes_{\alpha,\mu} G }
$$
from which it follows that the image of
$$
\C\rtimes_\mu G\to A\rtimes_{\alpha,\mu}G
$$
is the reduced group \cstar{}algebra.  Hence $\C\rtimes_\mu G$ dominates $C^*_\red(G)$.  Such an $(A,\alpha)$ exists whenever $G$ is discrete and exact: indeed, one may take $A=C(M)$ where $M$ is a minimal subsystem of the Stone-\v{C}ech compactification of $G$, as discussed in \cite[Sections 1.4-1.5]{Kerr}, and apply the result of \cite{Archbold-Spielberg} combined with the fact that the action of $G$ on $M$ is amenable, and thus $C(M)\rtimes_\un G=C(M)\rtimes_\red G$ \cite[Theorem 4.3.4]{Brown-Ozawa}.   Plausibly one could adapt such an argument to exact locally compact groups, but the necessary ingredients seem to be missing from the literature.  It is not known whether an $A$ with the above properties can exist for non exact groups $G$.

\section{The minimal exact correspondence functor}

One of the main motivations for considering exotic crossed products is to better understand counterexamples to the famous Baum-Connes conjecture.  The conjecture, with coefficients, in its original form claimed that a certain {\em  assembly map}
$$\as_{(G,A)}^\red: K_*^{\operatorname{top}}(G;A)\to K_*(A\rtimes_{\red}G)$$
should always be an isomorphism. We refer to \cite{Baum-Connes-Higson:BC} for the definition of the
assembly map and the group $K_*^{\operatorname{top}}(G;A)$.
We will not try to summarise the conjecture here, but note that all known failures of the Baum-Connes conjecture for groups with coefficients \cite{Higson-Lafforgue-Skandalis} are essentially all down to failures of exactness as in the following definition.

\begin{definition}\label{exact}
A crossed product functor $\mu$ is exact if for any short exact sequence of $G$-\cstar{}algebras
$$
\xymatrix{ 0\ar[r] & I \ar[r] & A \ar[r] & B \ar[r] & 0 },
$$
the corresponding sequence of crossed products
\begin{equation}\label{ses}
\xymatrix{ 0\ar[r] & I\rtimes_{\mu} G \ar[r] & A\rtimes_{\mu} G \ar[r] & B\rtimes_{\mu} G \ar[r] & 0 }
\end{equation}
is still exact.
\end{definition}

The universal crossed product is always exact, but infamous examples due to Gromov \cite{Gromov:random} (and recently given quite a satisfactory treatment by Osajda and others \cite{Osajda}) show that the reduced crossed product can fail to be exact for some groups.  Groups for which the reduced crossed product is exact were first studied by Kirchberg and Wassermann \cite{Kirchberg-Wassermann}, who called them \emph{exact} groups.  The class of exact groups includes many interesting classes of groups, such as almost connected groups, discrete linear groups, amenable groups, and hyperbolic groups.

In the sequence in line \eqref{ses} above, the map to $B\rtimes_{\mu} G$ is always surjective (as its image contains the dense \Star{}subalgebra $C_c(G,B)$), and the composition of the two central \Star{}homomorphisms is zero (by functoriality).  Hence the potential failures of exactness are that the map $I\rtimes_{\mu} G\to A\rtimes_{\mu} G$ might not be injective, and that the kernel of the map $A\rtimes_{\mu} G\to B\rtimes_{\mu} G$ might properly contain the image of the map $I\rtimes_{\mu} G\to A\rtimes_{\mu} G$.  The first of these is just the ideal property that we already considered in the previous section.  The second however, is independent of any of the (other) properties we have considered so far: indeed, BG crossed products are always exact, while generally failing most of the other properties in the previous section. On the other hand, KLQ crossed products have all of the properties considered in the previous section, but can fail to be exact as the reduced crossed product can fail to be exact.

Baum, Guentner, and Willett proposed to `fix' the Baum-Connes conjecture (for groups, with coefficients) by replacing the reduced crossed product that is traditionally used to define the conjecture with an exotic crossed product that is automatically exact.
Indeed, there is a canonical way to construct an assembly
\begin{equation}\label{eq-assembly}
\as_{(G,A)}^\mu: K_*^{\operatorname{top}}(G;A)\to K_*(A\rtimes_{\mu}G)
\end{equation}
for any crossed-product functor $\rtimes_\mu$, since the original map always factors over the $K$-theory $K_*(A\rtimes_{\un}G)$
of the universal crossed product. It is well known that the assembly map for the maximal crossed product, which is exact,
fails to be an isomorphism in general. So the exact crossed-product functor for the reformulated conjecture
should be as close to the reduced one as
possible.
For compatibility with Morita equivalences (and also to ensure the existence of a descent functor in $E$-theory) they require their functor to in addition have the following property.

\begin{definition}\label{mor com}
Let $\mathcal{K}_G$ denote the compact operators on $L^2(G)\otimes l^2(\N)$, equipped with the adjoint action $\Lambda$ coming from the tensor product of the regular representation and the trivial representation on $l^2(\N)$.  As for any $G$-\cstar{}algebra $(A,\alpha)$, the $G$-\cstar{}algebras $(A\otimes \mathcal{K}_G,\alpha\otimes \Lambda)$ and $(A\otimes \mathcal{K}_G,\alpha\otimes \text{id})$ are equivariantly Morita equivalent, strong Morita compatibility of the universal crossed product can be used to give a canonical isomorphism
$$
(A\otimes \mathcal{K}_G)\rtimes_{\alpha\otimes\Lambda,\un}G\cong (A\rtimes_{\alpha,\un}G)\otimes \mathcal{K}_G
$$
(see \cite[Corollary 5.4]{Buss-Echterhoff-Willett}).  A crossed product $\mu$ is \emph{Morita compatible} if this descends to an isomorphism
$$
(A\otimes \mathcal{K}_G)\rtimes_{\alpha\otimes\Lambda,\mu}G\cong (A\rtimes_{\alpha,\mu}G)\otimes \mathcal{K}_G.
$$
\end{definition}

In order to define the exotic crossed products used in the reformulations of the Baum-Connes conjecture, note that there is a natural order on the collection of all crossed products defined by saying $\rtimes_\mu\geq \rtimes_\nu$ if the identity on $C_c(G,A)$ extends to a \Star{}homomoprhism $A\rtimes_\mu G \to A\rtimes_\nu G$ for all $G$-\cstar{}algebras $A$.  Using an idea of Kirchberg, Baum, Guentner and Willett prove the following theorem.

\begin{theorem}\label{min ex}
There is a minimal exact and Morita compatible crossed product functor $\rtimes_\E$ with respect to the order above.

This crossed product can be used to reformulate the Baum-Connes conjecture with coefficients, asking the assembly map
$\as_{(G,A)}^\E$  of (\ref{eq-assembly}) to be an isomorphism for all $G$-algebras $(A,\alpha)$, in such a way that the reformulated conjecture has no (at time of writing!) known counterexamples, and such that some of the counterexamples to the old conjecture are confirming examples for the reformulated conjecture.
\end{theorem}

The minimal crossed product is natural to consider here as it is in some sense closest to the reduced crossed product, and as it does not change the conjecture for exact groups.  It also has the advantage that it does not suffer from the property (T) obstructions to the version of the Baum-Connes conjecture defined using the universal crossed product \cite[Corollary 5.7]{Baum-Guentner-Willett}.

Using the results of Section \ref{cp props sec}, one can prove an analogue of this result in the setting of correspondence functors \cite[Section 8]{Buss-Echterhoff-Willett}.

\begin{theorem}\label{min cor}
There is a minimal exact correspondence functor $\rtimes_{\E_\Cor}$ with respect to the order above.

Moreover, this functor agrees with the minimal exact Morita compatible functor on the category of separable $G$-\cstar{}algebras and equivariant \Star{}homomorphisms.
\end{theorem}

Combined with the results of the previous section, this shows that one can also use the full power of $KK^G$ theory to study the reformulated Baum-Connes conjecture.

\section{Restriction, extension, and induction to and from subgroups}\label{sec:indres}
Suppose that $H$ is a closed subgroup of $G$. In this section we want to study relations between crossed-product
functors on $G$ and crossed-product functors on $H$. We shall define in particular a restriction and extension
process between functors for $G$ and functors for $H$.

We start with the restriction process: Suppose that $\rtimes_\mu$ is a crossed-product functor for $G$
and suppose that $(A,\alpha)$ is an $H$-algebra. Consider the induced $G$\nb-algebra $(\Ind_H^G(A,\alpha), \Ind\alpha)$
in which
$$\Ind_H^G(A,\alpha):=\left\{F\in C_b(G,A): \begin{matrix} \alpha_h(F(sh))=F(s)\;\forall s\in G, h\in H,\\
\text{and}\; (sH\mapsto \|F(s)\|)\in C_0(G/H)\end{matrix}\right\}.$$
The $G$\nb-action on $\Ind_H^G(A,\alpha)$ is given by $\big(\Ind\alpha_s(F)\big)(t)=F(s^{-1}t)$.
Now recall Green's imprimitivity theorem (see \cite[{Theorem B2}]{Echterhoff-Kaliszewski-Quigg-Raeburn:Categorical} or
\cite{Williams:Crossed}), which provides a
natural equivalence bimodule $X(A,\alpha)$ between $\Ind_H^G(A,\alpha)\rtimes_{\Ind\alpha,\un}G$ and $A\rtimes_{\alpha,\un}H$.
Let
$$I_{\Ind\alpha,\mu}: =\ker\Big(\Ind_H^G(A,\alpha)\rtimes_{\Ind\alpha,\un}G\to \Ind_H^G(A,\alpha)\rtimes_{\Ind\alpha,\mu}G\Big).$$
By the Rieffel correspondence between ideals in
$\Ind_H^G(A,\alpha)\rtimes_{\Ind\alpha,\un}G$ and ideals in $A\rtimes_{\alpha,\un}H$ there is a unique ideal $I_{\alpha,\mu|_H}\subseteq A\rtimes_{\alpha,\un}H$ such that
$X(A,\alpha)$ factors through an equivalence bimodule
$X_\mu(A,\alpha)$ between $\Ind_H^G(A,\alpha)\rtimes_{\Ind\alpha,\mu}G$ and the quotient
\begin{equation}\label{eq-mures}
A\rtimes_{\alpha,\mu|_H}H:=(A\rtimes_{\alpha,\un}H)/I_{\alpha,\mu|_H}.
\end{equation}

\begin{definition}\label{def-mures}
Let $\rtimes_{\mu}$ be a crossed-product functor for $G$.
Then the assignment $(A,\alpha)\mapsto A\rtimes_{\alpha,\mu|_H}H$, with $ A\rtimes_{\alpha,\mu|_H}H$ constructed as
above, is called the {\em restriction} of $\rtimes_\mu$ to $H$.
\end{definition}

\begin{remark}\label{rem-Green}
By the definition of the restricted crossed-product functor the following version of
Green's imprimitivity theorem holds automatically: If $\rtimes_{\mu}$ is a $G$\nb-crossed-product functor and if $\rtimes_{\mu|_H}$ denotes
 its restriction to the closed subgroup $H$ of $G$, then Green's bimodule $X(A,\alpha)$ factors through
an $\Ind_H^G(A,\alpha)\rtimes_{\Ind\alpha,\mu}G-A\rtimes_{\alpha,\mu|_H}H$ equivalence bimodule.
Moreover, this determines the restriction $\rtimes_{\mu|_H}$ and since Green's imprimitivity theorem holds for both full and reduced norms (see \cite[Appendix B]{Echterhoff-Kaliszewski-Quigg-Raeburn:Categorical} for a detailed discussion and references), it follows that $\rtimes_{\un|_H}=\rtimes_\un$ and $\rtimes_{\red|_H}=\rtimes_\red$.
More generally, it follows from \cite[{Theorem~5.12}]{Buss-Echterhoff:Imprimitivity} that if  $\mu=\mu_E$ is
 a KLQ-crossed-product functor for $G$ corresponding to a $G$-invariant  ideal $E\subseteq B(G)$,
 then the restriction $\mu|_H$ is the KLQ-functor for $H$ which corresponds to the $H$-invariant
 ideal $E_H$ of $B(H)$ which is generated by $E|_H=\{f|_H: f\in E\}$.
\end{remark}

\begin{theorem}\label{thm-mures}
Let $\rtimes_{\mu}$ be a crossed-product functor for $G$ and let $\rtimes_{\mu|_H}$ be its restriction to $H$.
Then the following are true:
\begin{enumerate}
\item $(A,\alpha)\mapsto A\rtimes_{\alpha,\mu|_H}H$ is a crossed-product functor for $H$.
\item If $\rtimes_\mu$ has the ideal property, the same holds for $\rtimes_{\mu|_H}$.
\item If $\rtimes_\mu$ is a correspondence functor, the same holds for $\rtimes_{\mu|_H}$.
\item If $\rtimes_\mu$ is exact, the same holds for $\rtimes_{\mu|_H}$.
\end{enumerate}
\end{theorem}

\begin{remark}\label{rem-pull-back}
Before we give the proof of the theorem, we need to say some words about the connection of
the composition $\psi\circ \phi$ of two \Star{}homomorphisms $\phi:A\to B$ and $\psi:B\to C$
 and composition in the correspondence category $\Cor$.

For this observe that $\phi$ and $\psi$ can be represented by the (nondegenerate) $A-B$ and $B-C$ correspondences
$(\phi(A)B, \phi)$ and $(\psi(B)C,\psi)$, respectively. The assignment sending \Star{}homomorphisms to
the equivalence classes of these correspondences is functorial, so that the composition $\psi\circ\phi$ is represented
by the correspondence $(\phi(A)B\otimes_{B}\psi(B)C, \phi\otimes 1)$.
\end{remark}

\begin{proof} Suppose that $\phi:A\to B$ is a \Star{}homomorphism. In order to show that
it induces a \Star{}homomorphism $\phi\rtimes_{\mu|_H}H: A\rtimes_{\alpha,\mu|_H}H\to B\rtimes_{\beta,\mu|_H}H$
we need to show that the composition of the quotient map $q_{B,\mu|_H}: B\rtimes_{\beta,\un}H\to B\rtimes_{\beta,\mu|_H}H$
with the descent $\phi\rtimes_\un H: A\rtimes_{\alpha,\un}H\to B\rtimes_{\beta,\un}H$
factors through $A\rtimes_{\alpha,\mu|_H}H$.

For this we consider the following diagram in the correspondence category $\Cor$:
$$
\begin{CD}
A\rtimes_{\alpha,\un} H @> X(A,\alpha)^* >> \Ind_H^G(A,\alpha)\rtimes_{\Ind\alpha,\un}G\\
@V \phi\rtimes_{\un}H VV                         @VV \Ind_H^G\phi\rtimes_{\un} G V\\
B\rtimes_{\beta,\un}H @> X(B,\beta)^* >>  \Ind_H^G(B,\beta)\rtimes_{\Ind\beta,\un}G\\
@V q_{\mu|_H} VV   @VV q_\mu V\\
B\rtimes_{\beta,\mu|_H}H @> X_\mu(B,\beta)^* >>  \Ind_H^G(B,\beta)\rtimes_{\Ind\beta,\mu}G
\end{CD}
$$
It is shown in \cite[Chapter 4]{Echterhoff-Kaliszewski-Quigg-Raeburn:Categorical} that the upper square of this diagram commutes
(in \cite{Echterhoff-Kaliszewski-Quigg-Raeburn:Categorical} only reduced crossed products are considered,
but the same arguments work for full crossed products as well).
It follows from the definition of $\rtimes_{\mu|_H}$ that the lower square commutes.
Hence the outer square commutes as well. This implies that the kernel $J_{\Ind\phi\rtimes_\mu G}$
of the composition of the
right vertical arrows in the diagram corresponds to the kernel $J_{\phi\rtimes_{\mu|_H}H}$ of the
composition of the left vertical arrows via the Rieffel correspondence induced from $X(A,\alpha)$.
By assumption we have $J_{\Ind\phi\rtimes_\mu G}\subseteq I_{\Ind\alpha,\mu}$. Since
the Rieffel correspondence preserves inclusions, it follows that
$J_{\phi\rtimes_{\mu|_H}H}\subseteq I_{\alpha, \mu|_H}$, hence the left vertical arrows factor
through $A\rtimes_{\alpha,\mu|_H}H$. This proves (1).

Assume now that $I$ is an $H$-invariant ideal of $A$ with inclusion map $\iota:I\to A$
and let $Q:A\to A/I$ denote the quotient map. It is well known (and trivial to check) that the
corresponding sequence
$$0\longrightarrow \Ind_H^G(I,\alpha)\stackrel{\Ind\iota}{\longrightarrow}\Ind_H^G(A,\alpha)
\stackrel{\Ind Q}{\longrightarrow} \Ind_H^G(A/I,\alpha)\longrightarrow 0$$
is a short exact sequence of $G$\nb-algebras.
We then get the following commutative diagram in $\Cor$:
$$
\begin{CD}
I\rtimes_{\alpha,\mu|_H} H @> X_\mu(I,\alpha)^* >> \Ind_H^G(I,\alpha)\rtimes_{\Ind\alpha,\mu}G\\
@V \iota\rtimes_{\mu|_H}H VV                         @VV  \Ind_H^G\iota\rtimes_{\mu} G V\\
A\rtimes_{\alpha,\mu|_H}H @> X_\mu(A,\beta)^* >>  \Ind_H^G(A,\alpha)\rtimes_{\Ind\alpha,\mu}G\\
@V Q\rtimes_{\mu|_H}H VV                         @VV \Ind_H^GQ\rtimes_{\mu} G V\\
(A/I)\rtimes_{\alpha,\mu|_H}H @> X_\mu(A/I,\beta)^* >>  \Ind_H^G(A/I,\alpha)\rtimes_{\Ind\alpha,\mu}G
\end{CD}
$$
If $\rtimes_\mu$ has the ideal property, $\Ind_H^G\iota\rtimes_{{\mu}}G$ is injective. It follows from the commutativity of the upper square that ${\iota}\rtimes_{\mu|_H}H$ is injective as well. Hence
$\rtimes_{\mu|_H}$ also satisfies the ideal property. If $\rtimes_\mu$ is exact, then the
commutativity of the  whole diagram implies that $\rtimes_{\mu|_H}$ is exact.
This proves (2) and (4).

Finally, the commutativity of the diagram
$$
\begin{CD}
B\rtimes_{\alpha,\mu|_H} H @> X_\mu(I,\alpha)^* >> \Ind_H^G(B,\alpha)\rtimes_{\Ind\alpha,\mu}G\\
@V \iota\rtimes_{\mu|_H}H VV                         @VV  \Ind_H^G\iota\rtimes_{\mu} G V\\
A\rtimes_{\alpha,\mu|_H}H @> X_\mu(A,\beta)^* >>  \Ind_H^G(A,\alpha)\rtimes_{\Ind\alpha,\mu}G
\end{CD}
$$
where $B$ is an $H$-invariant hereditary subalgebra of $A$ -- which makes $\Ind_H^G(B,\alpha)$ a $G$\nb-invariant
hereditary subalgebra of $\Ind_H^G(A,\alpha)$ -- and where $\iota:B\to A$ denotes the inclusion map,
implies that the hereditary subalgebra property passes from $\rtimes_\mu$ to $\rtimes_{\mu|_H}$.
By Theorem~4.9 in \cite{Buss-Echterhoff-Willett} this implies (3) and finishes the proof.
\end{proof}

As a direct consequence of the above result, we get the following well-known, but non-trivial result (the original
proof by Kirchberg and Wassermann in \cite{Kirchberg-Wassermann:Permanence} uses similar ideas as used in the above
theorem):

\begin{corollary}
Every closed subgroup of an exact group is also exact.
\end{corollary}
\begin{proof}
A locally compact group $G$ is exact if and only if the minimal exact correspondence functor $\rtimes_{\E^G_\Cor}$
equals the reduced $G$-crossed-product functor $\rtimes_\red^G$. But this implies that if $H\sbe G$ is a
closed subgroup, then $\rtimes_\red^G|_H=\rtimes_\red^H$ is exact by Theorem~\ref{thm-mures}(4).
\end{proof}

The restriction of a crossed-product functor to a subgroup also has good properties
with respect to the Baum-Connes assembly map:

\begin{proposition}\label{prop-BCrestrict}
Suppose that $\rtimes_{\mu}$ is a crossed-product functor
for a second countable locally compact group $G$ and let $H$ be a closed subgroup of $G$.
Let $(A,\alpha)$ be any separable $H$-algebra. Then
the following are equivalent:
\begin{enumerate}
\item The assembly map $\as^{\mu|_H}_{(A,H)}: K_*^{\top}(H,A)\to K_*(A\rtimes_{\alpha,\mu|_H}H)$ is an isomorphism.
\item the assembly map $\as^{\mu}_{(\Ind_H^GA,G)}: K_*^{\top}(G, \Ind_H^GA)\to K_*(\Ind_H^GA\rtimes_{\Ind\alpha,\mu}G)$
is an isomorphism.
\end{enumerate}
In particular, if $G$ satisfies BC for all $\rtimes_{\mu}$-crossed products, then $H$ satisfies
BC for all $\rtimes_{\mu|_H}$-crossed products.
\end{proposition}
\begin{proof} It is shown in \cite[Theorem 2.2]{Chabert-Echterhoff:Permanence} that there is an
isomorphism $\Ind_H^G: K_*^{\top}(H,A)\to K_*^{\top}(G; \Ind_H^GA)$ which by
\cite[Proposition 2.3]{Chabert-Echterhoff:Permanence} commutes with the assembly maps
for $H$ and $G$ in the sense that the following diagram commutes:
$$
 \begin{CD}
 K_*^{\top}(H,A) @>\as^{\red}>>  K_*(A\rtimes_{\alpha,\red}H)\\
 @V\Ind_H^G VV      @VV \sim_M V\\
 K_*^{\top}(G, \Ind_H^GA) @>>\as^{\red}> K_*(\Ind_H^GA\rtimes_{\Ind\alpha,\red}G),
 \end{CD}
$$
where the right vertical arrow is induced from Green's Morita equivalence.
But the arguments used in the proof of \cite[Proposition 2.3]{Chabert-Echterhoff:Permanence}
can easily be adapted to show that the following diagram
$$
 \begin{CD}
 K_*^{\top}(H,A) @>\as^{\mu|H}>>  K_*(A\rtimes_{\alpha,\mu|_H}H)\\
 @V\Ind_H^G VV      @VV \sim_M V\\
 K_*^{\top}(G, \Ind_H^GA) @>>\as^{\mu}> K_*(\Ind_H^GA\rtimes_{\Ind\alpha,\mu}G).
 \end{CD}
$$
commutes as well. This finishes the proof.
\end{proof}

In view of the above proposition it is interesting to study the question whether
the restriction of the minimal exact correspondence functor for $G$ to a closed subgroup $H$ will
always give the minimal exact correspondence functor for $H$, since this would then imply
that the reformulated conjecture passes to closed subgroups. We shall show in the following section that this
is indeed true whenever $H$ is normal in $G$, but we do not know the answer in general.

We are now going
 to construct crossed-product functors for $G$ out of crossed-product functors for a closed subgroup $H$.
There are actually (at least) two possibilities for doing this. We start with what we call the {\em extension}
of a crossed-product functor to $G$:

\begin{definition}\label{def-ext}
Suppose $H$ is a closed subgroup of $G$ and let $(B,\beta)\mapsto B\rtimes_{\beta,\nu} H$ be a crossed-product functor
for $H$. Then, if $(A,\alpha)$ is a $G$\nb-algebra, we define the crossed product $A\rtimes_{\alpha, \ext\nu}G$ as
$A\rtimes_{\alpha, \ext\nu}G:=(A\rtimes_{\alpha,\un}G)/J_\nu$ with
$$J_\nu=\cap\{\ker(\pi\rtimes u): (\pi,u) \in \Rep(A\rtimes_{\alpha,\un}G)\;\text{such that}\; \pi\rtimes u|_H \in \Rep(A\rtimes_{\alpha,\nu}H)\}.$$
We call $\rtimes_{\ext\nu}$ the {\em extension} of $\rtimes_\nu$ to $G$.
\end{definition}

In other words, $A\rtimes_{\ext\nu}G$ is the ``largest'' $G$-crossed product such that all
representations of $A\rtimes_{\ext\nu}G$ restrict to representations of $A\rtimes_\nu H$. To get a feeling
for it observe that the extension of the universal crossed-product functor on $H$ is the universal
crossed-product functor for $G$, but the extension of the reduced crossed-product functor on $H$
will rarely be the reduced crossed-product functor for $G$. In fact, if $H$ is amenable, it will always be
the universal one.

\begin{theorem}\label{thm-ext}
Let $\rtimes_{\nu}$ be a crossed-product functor for $H$ and let $\rtimes_{\ext\nu}$ be the
extension of $\rtimes_\nu$ to $G$.
Then the following are true:
\begin{enumerate}
\item $(A,\alpha)\mapsto A\rtimes_{\alpha,\ext\nu}H$ is a crossed-product functor for $G$.
\item If $\rtimes_\nu$ has the ideal property, the same holds for $\rtimes_{\ext\nu}$.
\item If $\rtimes_\nu$ is a correspondence functor, the same holds for $\rtimes_{\ext\nu}$.
\item If $\rtimes_\nu$ is exact, the same holds for $\rtimes_{\ext\nu}$.
\end{enumerate}
\end{theorem}

\begin{proof}
In all four cases, we just show that the property of interest can be reformulated in terms of covariant pairs; having done this,
it is then straightforward to verify that if $\nu$ has the given property, then $\ext\nu$ also does.

For functoriality, let $\phi:A\to B$ be an $L$-equivariant \Star{}homomorphism for some group $L$.  Then for given quotients $A\rtimes_\mu L$ and $B\rtimes_\mu L$ of the universal crossed products, $\phi\rtimes_\un L:A\rtimes_\un L\to B\rtimes_\un L$ descends to a \Star{}homomorphism $A\rtimes_\mu L\to B\rtimes_\mu L$ if and only if for every covariant pair $(\pi,u)$ for $(B,L)$ that extends to $B\rtimes_{\mu} L$, the covariant pair $(\pi\circ \phi,u)$ extends to $A\rtimes_{\mu} L$.

For the ideal property, let $I$ be an $L$-invariant ideal in some $A$ and $\rtimes_\mu$ a crossed product functor for $L$.  For a nondegenerate representation $\pi$ of $I$, let $\widetilde{\pi}$ denote the canonical extension to $A$.  Note that $\rtimes_\mu$ has the ideal property if and only if $\widetilde{\pi}\rtimes u$ extends to $A\rtimes_{\mu}L$ whenever the covariant pair $(\pi,u)$ integrates to a representation of $A\rtimes_{\mu}L$.

To check the correspondence functor property, we work with the projection property.   Let $A$ be an $L$-algebra and $p$ an $L$-invariant projection in the multiplier algebra of $A$.  Let $\rtimes_\mu$ be a crossed product.  For a nondegenerate representation $\pi$ of $A$ on a Hilbert space $\mathcal{H}$, let $\pi|_p$ denote the restriction of $\pi$ to the corner $pAp$ acting on $\pi(p)\mathcal{H}$ (where we have also used $\pi$ for the canonical extension of $\pi$ to the multiplier algebra of $A$).  Recall from \cite[Corrollary 8.6]{Buss-Echterhoff-Willett} that the crossed product $\rtimes_{\mu}$ has the projection property if and only if for any such $A$ and $p$ and any covariant pair $(\pi,u)$ that integrates to $A\rtimes_{\mu} L$, the covariant pair $(\pi|_p,u)$ for $(pAp,L)$ integrates to $pAp\rtimes_{\mu}L$.

Finally, let
$$
0\to I \to A \to B \to 0
$$
be a short exact sequence of $L$-\cstar{}algebras and $\rtimes_\mu$ a crossed product. As we have already considered the ideal property, it remains to characterise exactness of the sequence
$$
0\to I\rtimes_{\mu}L \to A\rtimes_{\mu}L\to B\rtimes_{\mu}L \to 0
$$
at the middle term. For a representation $\pi$ of $A$ that contains $I$ in its kernel, let $\dot{\pi}$ denote the representation of $B$ canonically induced by $\pi$. Note that the sequence above is exact at the middle term precisely when for any representation $(\pi,u)$ that integrates to $A\rtimes_{\mu}L$ such that $\pi$ contains $I$ in the kernel, the representation $(\dot{\pi},u)$ integrates to $B\rtimes_{\mu}L$.

As a sample, we give the proof of (3) and leave the other assertions to the reader. For this let $p\in \M(A)$ be a $G$-invariant projection and let $(\pi,u)$ be a covariant representation of $(A, G,\alpha)$ that integrates to $A\rtimes_{\ext\nu}G$.
Then, by definition of $\rtimes_{\ext\nu}$,  $(\pi, u|_H)$ integrates to $A\rtimes_{\nu}H$.
The projection property for $\rtimes_\nu$ implies that $(\pi|_p, u|_H)$ integrates to $pAp\rtimes_\nu H$, which then implies that
$(\pi|_p, u)$ integrates to $pAp\rtimes_{\ext\nu}G$.
\end{proof}
We are now going to describe an alternative procedure to construct a crossed-product functor on $G$ from a functor
$\rtimes_\nu$ on a closed subgroup $H$ of $G$ which we call the {\em induced} crossed-product functor.
For this recall that if we start with a $G$\nb-algebra $(A,\alpha)$ and restrict the action to $H$, then the
induced algebra $\Ind_H^G(A,\alpha)$ is $G$\nb-isomorphic to $A\otimes C_0(G/H)$ equipped with the
diagonal action, where $G$ acts on $G/H$ by the left-translation action $\tau$ (e.g., see
\cite[Remark 6.1]{Echterhoff:Crossed}). Then Green's imprimitivity theorem
provides us with an equivalence bimodule $X(A,\alpha)$ between $A\rtimes_{\alpha,\un}H$ and
$(A\otimes C_0(G/H))\rtimes_{\alpha\otimes\tau, \un}G$. Given a crossed-product functor
$\rtimes_\nu$ for $H$ with corresponding ideal $I_{\alpha,\nu}\subseteq A\rtimes_{\alpha,\un}H$, we can
consider the ideal $I_{A\otimes C_0(G/H), \tilde\nu}$ in $(A\otimes C_0(G/H))\rtimes_{\alpha\otimes\tau, \un}G$
which corresponds to $I_{A,\nu}$ via the Rieffel correspondence for the
bimodule $X(A,\alpha)$. Let
$(A\otimes C_0(G/H))\rtimes_{\alpha\otimes \tau,\tilde\nu}G$ denote the quotient of the full crossed product by this ideal.
Let $j_A:A\to \M(A\otimes C_0(G/H))$ denote the canonical inclusion and let
$$q_{\tilde\nu}:(A\otimes C_0(G/H))\rtimes_{\alpha\otimes\tau, \un}G
\to (A\otimes C_0(G/H))\rtimes_{\alpha\otimes \tau,\tilde\nu}G$$
denote the quotient map. We then introduce the following:

\begin{definition}\label{def-induced-functor}
Let $H$ be a closed subgroup of $G$ and let $\rtimes_\nu$ be a crossed-product functor for $H$.
Then the {\em induced crossed-product functor} $\rtimes_{\Ind\nu}$ for $G$ is defined as
$$A\rtimes_{\alpha,\Ind\nu}G:=q_{\tilde\nu}\circ (j_A\rtimes_{\un} G)\big(A\rtimes_{\alpha,\un}G\big)
\subseteq \M\big((A\otimes C_0(G/H))\rtimes_{\alpha\otimes \tau,\tilde\nu}G\big).$$
\end{definition}

\begin{example}\label{exa:induced-functors}
The crossed-product functor induced from the reduced $H$-crossed product $\rtimes_\red^H$ is the reduced $G$-crossed product $\rtimes_\red^G$.
This is because Green's imprimitivity bimodule $X_H^G(A,\alpha)$ factors through the reduced norms
and the canonical homomorphism $A\rtimes_\red G\to \M((A\otimes \contz(G/H))\rtimes_\red G)$ is an embedding for every $G$-algebra $A$.

In particular, if $H$ is amenable (and in particular if $H$ is the trivial group), we always get the reduced $G$-crossed product functor by induction from the (unique) $H$-crossed product functor. This also means that induction from the universal norm does not always give the universal norm.
\end{example}

We have the following general properties for the induced crossed-product functors.

\begin{theorem}\label{thm-nuind}
Let $\rtimes_{\nu}$ be a crossed-product functor for $H$.
Then the following are true:
\begin{enumerate}
\item $(A,\alpha)\mapsto A\rtimes_{\alpha,\Ind\nu}G$ is a crossed-product functor for $G$.
\item If $\rtimes_\nu$ has the ideal property, the same holds for $\rtimes_{\Ind\nu}$.
\item If $\rtimes_\nu$ is a correspondence functor, the same holds for $\rtimes_{\Ind\nu}$.
\item If $\rtimes_\nu$ is exact and $G/H$ is compact, then $\rtimes_{\Ind\nu}$ is exact as well.
\end{enumerate}
\end{theorem}
\begin{proof}
Exactly the same arguments as used in the proof of Theorem~\ref{thm-mures} show that
the assignment $(A,H,\alpha)\mapsto (A\otimes C_0(G/H))\rtimes_{\alpha\otimes\tau,\tilde\nu}G$ is a
functor from the category of $G$\nb-algebras into the category of \cstar{}algebras which preserves suitable
versions of the ideal property or exactness, if this holds for the given functor on $H$. Then similar arguments as used in the
proof of \cite{Buss-Echterhoff-Willett}*{Corollary~4.20} imply items (1), (2) and (3). For the proof of (4) we consider the diagram
$$
\begin{CD}
 I\rtimes_{\nu}H @>>> A\rtimes_{\nu}H @>>> A/I\rtimes_\nu H \\
@V\sim_M V\cong V   @V\sim_MV\cong V @V\cong V\sim_MV\\
C(G/H, I)\rtimes_{\tilde\nu}G @>>> C(G/H, A)\rtimes_{\tilde\nu}G @>>> C(G/H, A/I)\rtimes_\nu G \\
@A\iota\rtimes GAA   @A\iota\rtimes GAA @AA\iota\rtimes GA\\
I\rtimes_{\Ind\nu}G @>>> A\rtimes_{\Ind\nu}G @>>> A/I\rtimes_{\Ind\nu} G
\end{CD}
$$
in which $\iota: A\to C(G/H,A)$ maps $a\in A$ to $a\cdot 1_{G/H}$.
The upper horizontal row is exact by exactness of $\rtimes_{\nu}$ and then the middle horizontal row
is exact since the upper half of the diagram commutes in the correspondence category (compare with
the proof of Theorem \ref{thm-mures}). But then a standard argument (see proof of Theorem~8.1 in \cite{Echterhoff:Crossed}),
using an approximate unit of $C(G/H, I)\rtimes_{\tilde\nu}G$ consisting of elements in $C_c(G,I)\subseteq I\rtimes_{\Ind\nu}G$
shows that the lower horizontal row is also exact. This finishes the proof.
\end{proof}

\begin{remark}\label{rem-induce}
If $H=\{e\}$ is the trivial subgroup of $G$, then the (unique) $G$\nb-crossed-product functor induced from $\{e\}$
is just the reduced crossed-product functor (see Example~\ref{exa:induced-functors}). In particular, the induced functor will not
be exact if $G$ is not exact, which shows that a functor induced from an
exact crossed-product functor does not have to be exact in general.
\end{remark}

\begin{lemma}\label{lem-indres}
Let $H$ be a closed subgroup of $G$ and let $\rtimes_{\mu}$ be a crossed-product functor on $G$.
Then $\rtimes_{\Ind(\mu|_H)}=\rtimes_{\mu_{C_0(G/H)}}$, where $\rtimes_{\mu_{C_0(G/H)}}$ is the
functor constructed from $\rtimes_\mu$ by tensoring with $D=C_0(G/H)$ as in \cite{Buss-Echterhoff-Willett}*{Corollary~4.20}.
\end{lemma}
\begin{proof}
Let $(A,\alpha)$ be a $G$-algebra. Then $\Ind_H^GA$ is isomorphic to $A\otimes C_0(G/H)$ via the isomorphism
which sends a function $F\in \Ind_H^GA$ to the function $(g\mapsto \alpha_g(F(g)))\in C_0(G/H,A)$.
It follows that $A\rtimes_{\mu|_H}H$ is the unique quotient of $A\rtimes_{\un}H$
such that Green's $(A\otimes C_0(G/H))\rtimes_\un G- A\rtimes_\un H$ equivalence bimodule factors through an
$(A\otimes C_0(G/H))\rtimes_\mu G- A\rtimes_{\mu|_H}H$ equivalence bimodule. But it follows then from the definition
of the induced functor $\rtimes_{\Ind(\mu|_H)}$ that $A\rtimes_{\Ind(\mu|_H)} G$ is precisely the
 quotient of $A\rtimes_\un G$ by the kernel of the canonical homomorphism $A\rtimes_\un G\to \M\big((A\otimes C_0(G/H))\rtimes_\mu G\big)$,
 hence it coincides with $A\rtimes_{\mu_{C_0(G/H)}} G$.
\end{proof}

\begin{corollary}\label{cor:Ind-Exact=Exact} Let $H$ be a cocompact closed subgroup of $G$. Then
$$\rtimes_{\Ind\E_\Cor^H}=\rtimes_{\E_\Cor^G}.$$
\end{corollary}
\begin{proof}
It follows from \cite{Buss-Echterhoff-Willett}*{Corollary~8.9} together with
Lemma~\ref{lem-indres} that the  functors $\rtimes_{\Ind(\E_\Cor^G|_H)}$ and $\rtimes_{\E_\Cor^G}$ coincide.
Since $\rtimes_{\E_\Cor^G|_H}$ is an exact correspondence functor for $H$, it dominates $\rtimes_{\E_\Cor^H}$. Hence
$\rtimes_{\E_\Cor^G}=\rtimes_{\Ind(\E_\Cor^G|_H)}$ dominates $\rtimes_{\Ind\E_\Cor^H}$. Since
$\rtimes_{\Ind\E_\Cor^H}$ is an exact correspondence functor by Theorem \ref{thm-nuind}, it dominates $\rtimes_{\E_\Cor^G}$,
hence both must be equal.
\end{proof}

Again, the above result can be used to yield the following well-known consequence:

\begin{corollary}
If a locally compact group $G$ contains an exact cocompact closed subgroup $H$, then $G$ is exact.
\end{corollary}
\begin{proof}
If $H$ is exact, then it means that $\rtimes_{\E_\Cor^H}=\rtimes_\red^H$ is the reduced crossed-product functor.
But since the induced $G$-crossed-product functor from $\rtimes_\red^H$ is $\rtimes_\red^G$ (see Example~\ref{exa:induced-functors}),
it follows from Corollary~\ref{cor:Ind-Exact=Exact} that $\rtimes_{\E_\Cor^G}=\rtimes_\red^G$, so $G$ is exact.
\end{proof}

\section{Normal subgroups}

In this section we want to show that if $N$ is a closed normal subgroup of $G$ and if $\rtimes_\E:=\rtimes_{\E_\Cor^G}$ denotes the
 minimal exact correspondence functor for $G$, then the restriction $\rtimes_{\E|_N}$ of
 $\rtimes_\E$ to $N$ is the minimal exact correspondence functor for $N$. Thus, as a consequence, it follows
 from Proposition \ref{prop-BCrestrict} that the validity of the reformulated version of the Baum-Connes conjecture
 due to Baum, Guentner, and Willett will pass to closed normal subgroups.
 In order to prove the result we need some preparations. As a first step, we show that the minimal
 crossed-product functor behaves well with respect to automorphisms of the group:

 \begin{lemma}\label{lem-automorphism} Let $\varphi:G\to G$ be an automorphism of the locally compact group $G$ and let $\alpha:G\to\Aut(A)$
 be an action. Let $\alpha^\varphi:G\to \Aut(A)$ be the action $\alpha^\varphi:=\alpha\circ \varphi^{-1}$ and let $\delta_\varphi\in (0,\infty)$
 be the module of $\varphi$ for the Haar measure of $G$, i.e., we have
 $\int_G \psi(g)\,dg=\delta_\varphi \int_G \psi(\varphi^{-1}(g))\,dg$
 for all $\psi\in C_c(G)$. Then the mapping
 $$\Phi: C_c(G, A, \alpha)\to C_c(G,A,\alpha^\varphi); \;\Phi(f)(g)=\delta_\varphi f(\varphi^{-1}(g))$$
 extends to \Star{}isomorphisms
 \begin{align*}
 \Phi_\un: &A\rtimes_{\alpha,\un}G\congto A\rtimes_{\alpha^\varphi, \un}G,\\
  \Phi_\red: &A\rtimes_{\alpha,\red}G\congto A\rtimes_{\alpha^\varphi, \red}G,\quad\text{and}\\
   \Phi_\E: &A\rtimes_{\alpha,\E}G\congto A\rtimes_{\alpha^\varphi, \E}G.
   \end{align*}
   \end{lemma}
   \begin{proof} It is straightforward to check that $\Phi$ is a \Star{}isomorphism between the dense
   \Star{}subalgebras $C_c(G,A, \alpha)$ and $C_c(G,A, \alpha^\varphi)$ of the respective crossed products (the
   entries $\alpha$ and $\alpha^\varphi$ indicate the different operations on $C_c(G,A)$).
   To see that it extends to an isomorphism $ \Phi_\un: A\rtimes_{\alpha,\un}G\congto A\rtimes_{\alpha^\varphi, \un}G$
   we just observe that we have a bijective correspondence between covariant representations
   of $(A,G,\alpha)$ and of $(A,G,\alpha^\varphi)$ given by $(\pi, u)\mapsto (\pi, u^\varphi)$ with
   $u^\varphi:=u\circ \varphi^{-1}$ such that $\pi\rtimes u^\varphi(\Phi(f))=\pi\rtimes u(f)$. This implies that
   $$\|\Phi(f)\|_\un=\sup_{(\pi, u)} \|\pi\rtimes u^\varphi(\Phi(f))\|=\sup_{(\pi,u)}\|\pi\rtimes u(f)\|=\|f\|_\un.$$

For the reduced crossed products
   recall from Definition \ref{uni red cp}  the construction of the regular representation $\Lambda= \pi\rtimes(\lambda\otimes 1)$
   of $(A,G,\alpha)$ on the Hilbert $A$-module $L^2(G,A)$ with
   $(\pi(a)\xi)(g)=\alpha_{g^{-1}}(a)\xi(g)$
 for $a\in A, t,g\in G$ and $\xi\in L^2(G,A)$. Similarly, the C$^*$-part $\pi^\varphi$ of the regular representation
 $\Lambda^\varphi=\pi^\varphi\rtimes (\lambda\otimes 1)$ of $(A,G,\alpha^\varphi)$ is given by
   $(\pi^\varphi(a)\xi)(g)=\alpha_{\varphi^{-1}(g^{-1})}(a)\xi(g)$.
An easy computation  then shows  that  the unitary operator
 $$W: L^2(G, A)\to L^2(G, A); \; W\xi(g)=\sqrt{\delta_{\varphi}}\xi(\varphi^{-1}(g))$$
 intertwines the regular representation $\Lambda^\varphi$ with
  $\Lambda_A\rtimes\Lambda_G^\varphi$. Hence, up to unitary equivalence, the
 above described correspondence of covariant representations of $(A,G,\alpha)$ and $(A,G,\alpha^\varphi)$
 sends $\Lambda$ to $\Lambda^\varphi$ which proves that $\Phi$ extends to an isomorphism
 of the reduced crossed products.

 To see that $\Phi$ also extends to an isomorphism $\Phi_\E: A\rtimes_{\alpha,\E}G\congto A\rtimes_{\alpha^\varphi, \E}G$
 we argue as follows: Let $\tilde\Phi:A\rtimes_{\alpha,\un}G\to A\rtimes_{\alpha^\varphi, \E}G$ denote the surjective
 \Star{}homomorphism given by composing $\Phi_{\un}$ with the quotient map
 $q^\varphi:A\rtimes_{\alpha^\varphi, \un}G\onto A\rtimes_{\alpha^\varphi, \E}G$.
 Let $A\rtimes_{\alpha,\E^\varphi}G:=(A\rtimes_{\alpha, \un}G)/\ker\tilde\Phi$.
 It is straightforward to check that $(A,\alpha)\mapsto A\rtimes_{\alpha, \E^\varphi}G$ is an
 exact correspondence functor for $G$, hence it must dominate $A\rtimes_{\alpha, \E}G$ in the sense
 that the identity on $C_c(G,A)$ induces a surjective \Star{}homomorphism
  $q: A\rtimes_{\alpha,\E^\varphi}G\onto A\rtimes_{\alpha,\E}G$. In other words, this shows that
  the inverse $\Phi^{-1}: C_c(G,A,\alpha^\varphi)\to C_c(G,A,\alpha)$ extends to a surjective
  \Star{}homomorphism $A\rtimes_{\alpha^\varphi, \E}G\onto A\rtimes_{\alpha,\E}G$.
   Conversely, if we replace $\varphi$ by its inverse
  $\varphi^{-1}$ and apply the above results to the action $\alpha^\varphi$, we see that
  $\Phi$ extends to a surjective \Star{}homomorphism
  $A\rtimes_{\alpha,\E}G\onto A\rtimes_{\alpha^\varphi,\E}G$. This combines to show that $\Phi$ extends to
  the desired \Star{}isomorphism $\Phi_\E$.
  \end{proof}

   We now want to extend the above result to automorphisms $\varphi:=(\varphi_A,\varphi_G)$ of the system
   $(A,G,\alpha)$. This means that $\varphi_A:A\to A$ is a \Star{}automorphism of $A$ and $\varphi_G:G\to G$ is an
   automorphism of $G$ such that
   \begin{equation}\label{eq-auto}
   \alpha_{\varphi_G(s)}(\varphi_A(a))=\varphi_A(\alpha_s(a))\quad \forall s\in G, a\in A.
   \end{equation}
Note that if $\alpha:G\to \Aut(A)$ is an action and $N\subseteq G$ is a normal
   subgroup of $G$, then every $g\in G$ determines an automorphism $\gamma_g:=(\alpha_g, C_g)$
   of $(A, N,\alpha)$ with $\alpha_g:A\to A$ the given action of the element $g\in G$ and $C_g:N\to N$
   the automorphism given by conjugation with $g$: $C_g(n)=gng^{-1}$.
   It is well-known that every automorphism $\varphi$ of $(A,G,\alpha)$ induces  automorphisms
   $\varphi_\un$ and $\varphi_\red$ on $A\rtimes_{\alpha,\un}G$ and $A\rtimes_{\alpha,\red}G$, respectively,
   both extending the \Star{}isomorphism  $\tilde\varphi:C_c(G,A)\to C_c(G,A)$ given by the formula
   \begin{equation}\label{eq-autom}
   \big(\tilde\varphi(f)\big)(s)=\delta_\varphi \varphi_A(f(\varphi_G^{-1}(s)), \quad f\in C_c(G,A), s\in G,
   \end{equation}
   where $\delta_\varphi$ denotes the module of the automorphism $\varphi_G$.
   The proof follows easily from Lemma \ref{lem-automorphism} and the arguments given in the proof of
   the following lemma in the case of the $\rtimes_\E$-crossed products:

   \begin{lemma}\label{lem-automorphisms2}
   If $\varphi=(\varphi_A,\varphi_G)$ is an automorphism of $(A,G,\alpha)$, then
   the  \Star{}isomorphism $\tilde\varphi:C_c(G,A)\to C_c(G,A)$ extends to a
   \Star{}automorphism $\varphi_\E$ of $A\rtimes_{\alpha, \E}G$.
    \end{lemma}
    \begin{proof}
    Let $\alpha^\varphi:=\alpha\circ \varphi_G^{-1}$ be the action of $G$ on $A$ `twisted' by $\varphi_G$.
     Lemma \ref{lem-automorphism} shows that
    $\Phi:C_c(G,A,\alpha)\to C_c(G,A,\alpha^\varphi)$ given by $\Phi(f)(s)=\delta_\varphi f(\varphi_G^{-1}(s))$
   extends to an isomorphism $\Phi_\E: A\rtimes_{\alpha, \E}G\congto A\rtimes_{\alpha^\varphi, \E}G$.
    On the other hand it follows from (\ref{eq-auto}) that the automorphism $\varphi_A:A\to A$ is
   $\alpha^\varphi-\alpha$ equivariant. By functoriality of $\rtimes_\E$ it therefore induces an isomorphism
   $\varphi_A\rtimes G: A\rtimes_{\alpha^\varphi, \E}G\to A\rtimes_{\alpha,\E}G$ given
   on $C_c(G,A)$ by sending a function $f$ to $\varphi_A\circ f$. The composition
   $\varphi_\E:=(\varphi_A\rtimes G)\circ \Phi_\E$ is then an automorphism of $A\rtimes_{\alpha, \E}G$
   which extends $\tilde\varphi: C_c(G,A)\to C_c(G,A)$.
    \end{proof}

\begin{example}\label{ex-autom}
We should note that the argument of the lemma works  for any crossed-product functor $\rtimes_\mu$
for $G$ such that an analogue of Lemma \ref{lem-automorphism} holds for $\rtimes_\mu$, i.e., the
isomorphism $\Phi: C_c(G, A,\alpha)\to C_c(G,A,\alpha^\varphi)$ of that lemma extends to an isomorphism
$\Phi_\mu: A\rtimes_{\alpha,\mu}G\congto A\rtimes_{\alpha^\varphi, \mu}G$. Hence it applies in particular
for the full and reduced crossed products.

But  we should point out that analogues of Lemmas \ref{lem-automorphism} and \ref{lem-automorphisms2}.
do not even hold for all KLQ crossed-product functors. For this let $G$ be any non-amenable
group and let $E\subseteq B(G\times G)$ be the
weak$^*$ closure of the ideal consisting
of all coefficient functions $\phi$  that are supported in a set of the form $G\times K$, where $K$ is a compact subset of $G$.
Then a unitary representation
$v:G\times G\to \U(\H)$  integrates to $C_E^*(G\times G)$  if and only if its restriction
to the second factor
is weakly contained in $\lambda_G$.
Let $\varphi:G\times G\to G\times G$ be the flip automorphism and let $v=u_G\otimes \lambda_G$
where $u_G$ denotes the universal representation of $G$. Then $v$ factors through
$C_E^*(G\times G)$ but $v\circ \varphi =\lambda_G\otimes u_G$ does not. Hence
$\varphi$ does not `extend' to an automorphism of $C_E^*(G\times G)$
and  the corresponding KLQ-functor fails Lemmas  \ref{lem-automorphism} and \ref{lem-automorphisms2}.
\end{example}

 \begin{lemma}\label{lem-normal1} Let $\rtimes_\mu$ be a crossed-product functor for $G$ and let
 $N$ be a closed normal subgroup of $G$. Further let $(A,G,\alpha)$ be a $G$-system.
 Then the action $\gamma^\mu$ of $G$ on $C_c(N,A)$ defined on the level of $C_c(N,A)$ by the formula
$$ (\gamma^\mu_g(f))(n)=\delta_g \alpha_g(f(g^{-1}ng)),$$
where $\delta_g$ denotes the module for the automorphism $C_g(n)=gng^{-1}$ of $N$,  extends to an action $\gamma^\mu:G\to \Aut(A\rtimes_{\alpha, \mu|_N}N)$.
 \end{lemma}
\begin{proof} Recall that $A\rtimes_{\alpha,\mu|_N}N$  is defined
as the quotient $(A\rtimes_{\alpha,\un}N)/J_\mu$, where $J_\mu$ is the ideal in
$A\rtimes_{\alpha,\un}N$ which corresponds to the ideal
$$K_\mu:=\ker(\Ind_N^GA\rtimes_{\Ind\alpha,\un}G\to \Ind_N^GA\rtimes_{\Ind\alpha,\mu}G)$$
via Green's imprimitivity bimodule. Since $A$ is a $G$-algebra, the induced algebra
$\Ind_N^GA$ is $G$-isomorphic to $A\otimes C_0(G/N)$ equipped with the diagonal action
$\alpha\otimes \tau$, where $\tau$ denotes the left-translation action.
Let $\beta_g$ denote the automorphism
of $(A\otimes C_0(G/N))\rtimes_{\alpha\otimes\tau, \un}G$ which is induced from
the action $\id_A\otimes \sigma_g$ of $g\in G$ on $A\otimes C_0(G/N)$, where
$\sigma_g$ denotes the right translation action on $C_0(G/N)$. Note that
$\beta_g$ exists since $\id_A\otimes \sigma_g$ commutes with $\alpha\otimes \tau$.
It follows then from functoriality of $\rtimes_{\mu}$ that $\beta_g$ factors through an
automorphism of $(A\otimes C_0(G/N))\rtimes_{\alpha\otimes \tau, \mu}G$. In particular, the
kernel $K_\mu$ of the quotient map $(A\otimes C_0(G/N))\rtimes_{\alpha\otimes\tau, \un}G\to (A\otimes C_0(G/N))\rtimes_{\alpha\otimes\tau, \mu}G$
is $\beta_g$-invariant.
Now it is part of  \cite[Theorem 1]{Echterhoff:Morita_twisted} that there is $\beta_g-\gamma_g^\un$-compatible
automorphism $\nu_g$ of $X(A,\alpha)$.
Since $J_\mu$ corresponds to $K_\mu$ via the Rieffel correspondence for the bimodule $X(A,\alpha)$,
it follows that $\beta_g(K_\mu)$ corresponds to $\gamma_g^\un(J_\mu)$ by Rieffel correspondence.
Since this correspondence is one-to-one, and since $K_\mu=\beta_g(K_\mu)$ we
also get $J_\mu=\gamma_g^\mu(J_\mu)$. Thus $\gamma_g^\un$ factors through
an automorphism $\gamma_g^\mu$ of $A\rtimes_{\alpha,\mu|_N}N$.
\end{proof}

 We now turn to the problem of showing that the restriction $\rtimes_{\E|_N}$ of the
 minimal exact correspondence functor $\rtimes_\E:=\rtimes_{\E^G_\Cor}$ to any closed
 normal subgroup $N$ of $G$ coincides with the minimal exact correspondence functor
 $\rtimes_{\E^N_\Cor}$ for $N$. We need

\begin{lemma}\label{lem-include}
Let $H$ be a closed subgroup of $G$ and $\rtimes_\nu$ a crossed-product functor on $H$.
Then, for any $G$-algebra $(A,\alpha)$, the canonical mapping
$i_A\rtimes i_G|_H: A\rtimes_\un H\to \M(A\rtimes_{\un}G)$ factors to a well-defined
\Star{}homomorphism
$$i_A^{\ext\nu}\rtimes i_G^{\ext\nu}|_H: A\rtimes_\nu H\to \M(A\rtimes_{\ext\nu}G).$$
If $H$ is open in $G$, this map takes its image in $A\rtimes_{\ext\nu}G$.
\end{lemma}
\begin{proof} It follows from the definition of $\rtimes_{\ext\nu}$ that
$i_A^{\ext\nu}\rtimes i_G^{\ext\nu}|_H$ is well defined and it
is clear that it takes image in $A\rtimes_{\ext\nu}G$ if $H$ is open in $G$.
\end{proof}

Just for the records we show as a corollary of this lemma and Theorem \ref{thm-ext} that the minimal exact correspondence
functor $\rtimes_{\E_\Cor}$ enjoys the following property for arbitrary closed subgroups which is
well-known for the maximal and reduced crossed-product functors:

\begin{corollary}\label{subgroups}
Let $H$ be a closed subgroup of $G$, and let $(A,\alpha)$ be a $G$-\cstar{}algebra.
Then the canonical mapping
$i_A\rtimes i_G|_H\colon A\rtimes_\un H\to \M(A\rtimes_{\un}G)$ factors to a well-defined \Star{}homomorphism
$$i_A^{\E_\Cor}\rtimes i_G^{\E_\Cor}|_H\colon A\rtimes_{\E_\Cor^H} H\to \M(A\rtimes_{\E_\Cor^G}G).$$
\end{corollary}
\begin{proof} This follows from Lemma \ref{lem-include} together with the fact (shown in Theorem \ref{thm-ext}) that
the extension $\rtimes_{\ext\E_\Cor^H}$ to $G$ is an exact correspondence functor and
hence dominates $\rtimes_{\E_\Cor^G}$.
\end{proof}

\begin{remark}
It is not clear whether the homomorphism
$$i_A^{\E_\Cor}\rtimes i_G^{\E_\Cor}|_H\colon A\rtimes_{\E_\Cor^H} H\to \M(A\rtimes_{\E_\Cor^G}G)$$
 is injective in general, even if $H$ is open in $G$.
\end{remark}

Suppose now that $\rtimes_\mu$ is a crossed-product functor for $G$
and let $N$ be a closed normal
subgroup of $G$. Then, if $(A,\alpha)$ is a $G$-algebra, we get a canonical \Star{}homomorphism
\begin{equation}\label{eq-homN}
i_A^\mu\rtimes i_G^\mu|_N: A\rtimes_{\alpha,\un}N\to \M((A\otimes C_0(G/N))\rtimes_{\alpha\otimes\tau, \mu}G),
\end{equation}
given as the composition
$$q_\mu\circ \big((\id_A\otimes 1_{C_0(G/N)})\rtimes_\un G\big)\circ (i_A^\un\rtimes i_G^\un|_N),$$
with
$i_A^\un\rtimes i_G^\un|_N: A\rtimes_{\alpha,\un}N\to \M(A\rtimes_{\alpha,\un}G)$ the canonical homomorphism,
\linebreak
{$(\id_A\otimes 1_{C_0(G/N)})\rtimes_\un G: A\rtimes_{\alpha,\un}G\to \M((A\otimes C_0(G/N))\rtimes_{\alpha\otimes\tau, \un}G)$}
induced from $\id_A\otimes 1_{C_0(G/N)}:  A\to \M(A\otimes C_0(G/N))$ by
functoriality for generalized homomorphism of $\rtimes_\un$, and
$q_\mu: (A\otimes C_0(G/N))\rtimes_{\alpha\otimes\tau, \un}G\to (A\otimes C_0(G/N))\rtimes_{\alpha\otimes\tau, \mu}G$
the quotient map.

\begin{lemma}\label{lem-normal2}
The homomorphism of (\ref{eq-homN}) factors to a faithful \Star{}homomorphism
$$i_A^\mu\rtimes i_G^\mu|_N: A\rtimes_{\alpha,\mu|_N}N\to \M((A\otimes C_0(G/N))\rtimes_{\alpha\otimes\tau, \mu}G).$$
\end{lemma}

\begin{proof}
Let $\rho\rtimes v$ be a covariant representation of $A\rtimes_{\alpha,\un}N$ which factors through a faithful representation of $A\rtimes_{\alpha,\mu|_N}N$.
Then the induced representation $\Ind_N^G(\rho\rtimes v)$ of $(A,G,\alpha)$ can be defined by first inducing $\rho\rtimes v$ to $(A\otimes C_0(G/N))\rtimes_{\alpha\otimes\tau, \un}G$ via Green's imprimitivity module $X(A,\alpha)$, which by assumption factors to a faithful representation $\Ind^{X(A,\alpha)}(\rho\rtimes v)$ of $(A\otimes C_0(G/N))\rtimes_{\alpha\otimes\tau, \mu}G$,
and then composing this representation with the canonical embedding
$$(\id_A\otimes 1_{C_0(G/N)})\rtimes_\un G: A\rtimes_{\alpha,\un}G\to \M((A\otimes C_0(G/N))\rtimes_{\alpha\otimes\tau, \un}G).$$
(e.g., see \cite[\S 6 and \S 9]{Echterhoff:Crossed} and in particular \cite[Remark 9.5]{Echterhoff:Crossed} for more details).
The restriction $\Res_N^G(\Ind_N^G(\rho\rtimes v))$ of $\Ind_N^G(\rho\rtimes v)$ to $A\rtimes_{\alpha,\un}N$
is then given by the composition $\Ind^{X(A,\alpha)}(\rho\rtimes v)\circ (\id_A\otimes 1)\rtimes_\un G\circ (i_A^\un\rtimes i_G^\un|_N)$.
Since $\Ind^{X(A,\alpha)}(\rho\rtimes v)$ factors through a faithful representation of
$(A\otimes C_0(G/N))\rtimes_{\alpha\otimes\tau, \mu}G$, we may identify it with the quotient map
$q_\mu$, and it follows that
$I_\mu:=\ker \big(\Res_N^G(\Ind_N^G(\rho\rtimes v))\big)$
coincides with the kernel
of (\ref{eq-homN}).
On the other hand, if $J_\mu:=\ker(\rho\rtimes v)=\ker\big(A\rtimes_{\alpha,\un}N\to A\rtimes_{\alpha,\mu|_N}N\big)$,
then by \cite[Proposition 9.15]{Echterhoff:Crossed} we get
\begin{align*}
I_\mu&=\ker(\Res_N^G(\Ind_N^G(\rho\rtimes v)))=\bigcap\{\gamma_g^\un(\ker(\rho\rtimes v)): g\in G\}\\
&=\bigcap\{\gamma_g^\un(J_\mu): g\in G\},
\end{align*}
where $\gamma_g^\un:G\to \Aut(A\rtimes_{\alpha,\un}N)$ is the decomposition action of $G$ on $A\rtimes_{\alpha,\un}N$.
Since this action factors through $A\rtimes_{\alpha,\mu|_N}N$ by Lemma \ref{lem-normal1}, it follows that
$\gamma_g^\un(J_\mu)=J_\mu$ for all $g\in G$. Thus  $I_\mu=J_\mu$. Since $I_\mu$ is the kernel of (\ref{eq-homN}),
the result follows.
\end{proof}

Recall that a $C_0(X)$-algebra is a \cstar{}algebra $A$ equipped with a nondegenerate
\Star{}homomorphism $\varphi:C_0(X)\to \ZM(A)$, where $X$ is a locally compact space and $\ZM(A)$ denotes the
centre of the multiplier algebra of $A$. Then, for each $x\in X$, the fibre $A_x$ of $A$ over $x$ is defined as the
quotient $A_x:=A/I_x$ with $I_x:=\varphi(C_0(X\setminus\{x\}))A$. For each $a\in A$ we get an assignment
$x\mapsto a_x:=a+I_x\in A_x$, hence we may regard $A$ as an algebra of sections of a bundle $\mathcal A$
of \cstar{}algebras over $X$ with fibres $A_x$.
An action $\alpha:G\to \Aut(A)$ is called {\em fibre-wise}, if $\alpha_g(I_x)=I_x$ for all $g\in G, x\in X$ and hence
the action induces actions $\alpha^x:G\to \Aut(A_x)$ for each $x\in X$ in a canonical way. Note that
$\alpha$ being fibre-wise is equivalent to $\alpha$ being $C_0(X)$-linear in the sense that $\alpha_g(\varphi(f)a)=\varphi(f)\alpha_g(a)$
for all $g\in G$ and $f\in C_0(X)$. For more information on $C_0(X)$-algebras we refer to \cite[Appendix C]{Williams:Crossed}.

\begin{lemma}\label{lem-field}
Suppose that $\alpha:G\to \Aut(A)$ is a fibre-wise action of $G$ on the $C_0(X)$-algebra $A$
with structure map $\varphi:C_0(X)\to \ZM(A)$. Suppose further that $\rtimes_\mu$ is an exact crossed-product functor for $G$.
Then $A\rtimes_{\alpha,\mu}G$ is a $C_0(X)$-algebra with structure map
$i_A\circ \varphi: C_0(X)\to \ZM(A\rtimes_{\alpha,\mu}G)$ and fibres $A_x\rtimes_{\alpha^x,\mu}G$.
\end{lemma}
\begin{proof}
As a composition of nondegenerate homomorphisms, $i_A\circ \varphi$ is well defined and nondegenerate, and
one easily checks on the generators of $A\rtimes_{\alpha,\mu}G$ that it maps $C_0(X)$
into $\ZM(A\rtimes_{\alpha,\mu}G)$. To see that $(A\rtimes_{\alpha,\mu}G)_x\cong A_x\rtimes_{\alpha^x,\mu}G$ for all $x\in X$
we first observe that for all open $U\subseteq X$ we get
$$i_A\circ \varphi\big(C_0(U)\big)(A\rtimes_{\mu}G)=(\varphi(C_0(U))A)\rtimes_{\mu}G,$$
 since by the ideal property (which follows
from exactness) both coincide with the closure of $C_c(G, \varphi(C_0(U))A)$ inside $A\rtimes_\mu G$.
Hence, if $U=X\setminus\{x\}$ and $I_x:=\varphi(C_0(X\setminus\{x\}))A$, it follows from exactness of $\rtimes_{\mu}$ that
$$A_x\rtimes_{\alpha^x,\mu}G=(A\rtimes_{\alpha,\mu} G)/(I_x\rtimes_{\alpha,\mu} G)=(A\rtimes_{\alpha,\mu} G)_x.$$
\end{proof}

\begin{remark}
The above lemma is not true for non-exact crossed products in general. Indeed, it is not true for the reduced crossed
product of a  non-exact group $G$ by \cite[Theorem on  p.~170]{Kirchberg-Wassermann}.
\end{remark}

As a corollary of Lemma~\ref{lem-field}, we get

\begin{lemma}\label{lem-fibre}
Suppose that $\rtimes_{\mu_i}$, $i=1,2$, are two exact crossed-product functors for $G$.
Suppose further that $A$ is a $C_0(X)$-algebra equipped with a fibre-wise action
$\alpha:G\to\Aut(A)$
such hat $A\rtimes_{\alpha,\mu_1}G= A\rtimes_{\alpha, \mu_2}G$.
Then $A_x\rtimes_{\alpha^x,\mu_1}G=A_x\rtimes_{\alpha^x,\mu_2}G$ for all $x\in X$.
\end{lemma}
\begin{proof} This follows from Lemma \ref{lem-field} together with the fact that
a $C_0(X)$-linear isomorphism between two $C_0(X)$-algebras always induces
isomorphisms of the fibres.
\end{proof}

We want to apply the above result to the following example: Suppose that $N\subseteq G$ is a closed normal subgroup of $G$ and
let $(B,\beta)$ be an $N$-\cstar{}algebra. Then the induced $G$-algebra $\Ind_N^GB$ becomes a $C_0(G/N)$-algebra with respect to the nondegenerate \Star{}homomorphism $\varphi: C_0(G/N)\to \ZM(\Ind_N^GB)$ given by
$$\big(\varphi(f) F\big)(g)=f(gN)F(g),\quad\forall f\in C_0(G/N), F\in \Ind_N^GB, g\in G.$$
Then each fibre $(\Ind_N^G B)_{gN}$ identifies with $B$ via the evaluation map $F\mapsto F(g)$.
It is easy to check that the restriction of
the $G$-action $\Ind\beta$ to $N$ is $C_0(G/N)$-linear, hence fibre-wise.
For the unit fibre $eN$, it follows that the $N$-action
on the fibre $B\cong (\Ind_N^GB)_{eN}$ coincides with
the action $\beta:N\to \Aut(B)$ we started with. This follows from the equation
$$(\Ind\beta_n(F))(e)=F(n^{-1})=\beta_n(F(e)),\quad \forall F\in \Ind_N^GB.$$
Using this, we get

\begin{lemma}\label{lem-normal3}
Suppose that $N$ is a closed normal subgroup of $G$ and that $\rtimes_{\nu_i}$, $i=1,2$, are
two exact crossed-product functors for $N$. Suppose further that $\rtimes_{\nu_1}=\rtimes_{\nu_2}$
when restricted to the category of $G$-\cstar{}algebras. Then $\rtimes_{\nu_1}=\rtimes_{\nu_2}$.
\end{lemma}
\begin{proof} Let $(B,\beta)$ be any $N$-algebra. Then, by assumption, we have
$\Ind_N^GB\rtimes_{\nu_1}N=\Ind_N^GB\rtimes_{\nu_2}N$. But then the above discussion
together with Lemma \ref{lem-fibre} implies that $B\rtimes_{\nu_1}N=B\rtimes_{\nu_2}N$.
\end{proof}

We are now ready for the main result of this section:

\begin{theorem}\label{thm-ext-normal}
Let $N$ be a closed normal subgroup of the locally compact group $G$.
Then the minimal exact correspondence functor
$\rtimes_{\E_\Cor^N}$ for $N$ is equal to the restriction $\rtimes_{\E^G_\Cor|_N}$ of the minimal
exact correspondence functor $\rtimes_{\E_\Cor^G}$ for $G$.
\end{theorem}
\begin{proof}
We show that the restriction of $\rtimes_{\ext \E_\Cor^N}$ to $N$ equals $\rtimes_{\E_\Cor^N}$. Since
$\rtimes_{\ext \E_\Cor^N}$ is an exact correspondence functor by Theorem \ref{thm-ext}
it dominates $\rtimes_{\E_\Cor^G}$. Hence the desired equation would imply that
the restriction of $\rtimes_{\E_\Cor^G}$ to $N$, which is an exact correspondence functor
by Theorem \ref{thm-mures}, would be dominated by $\rtimes_{\E_\Cor^N}$. But then they must
coincide.

By Lemma \ref{lem-normal3} it suffices to show that the restriction of
$\rtimes_{\ext \E_\Cor^N}$ to $N$ coincides with $\rtimes_{\E_\Cor^N}$
on every $G$-\cstar{}algebra $(A,\alpha)$. Let us write $\mu:=\ext \E_\Cor^N$.
Let $(i_A^\mu, i_{C_0(G/N)}^\mu, i_G^\mu)$ denote the canonical maps from $(A,C_0(G/N), G)$ into
$\M((A\otimes C_0(G/N))\rtimes_{\alpha\otimes \tau,\mu}G)$.
It follows then from Lemma \ref{lem-normal2} that $i_A^\mu\rtimes i_G^\mu|_N$ factors through a faithful
\Star{}homomorphism
\begin{equation}\label{eq-map1}
i_A^\mu\rtimes i_G^\mu|_N: A\rtimes_{\alpha,\mu|_N}N\to \M((A\otimes C_0(G/N))\rtimes_{\alpha\otimes \tau,\mu}G).
\end{equation}
On the other hand, we know from Lemma \ref{lem-include} that we have a
homomorphism
$$
(i_A^\mu\otimes i_{C_0(G/N)}^\mu)\rtimes i_G^\mu|_N:
(A\otimes C_0(G/N))\rtimes_{\alpha,\E_\Cor^N}N\to \M((A\otimes C_0(G/N))\rtimes_{\alpha\otimes \tau,\mu}G).
$$
Since $\rtimes_{\mu|_N}$ is an exact crossed-product functor, it follows then from minimality of $\rtimes_{\E_\Cor^N}$
that there is a surjective \Star{}homomorphism $A\rtimes_{\alpha,\mu|_N}N\to A\rtimes_{\alpha,\E_\Cor^N}N$.
Hence we get a chain of \Star{}homomorphisms
\begin{equation}\label{eq-composition}
\begin{split}
A\rtimes_{\alpha,\mu|_N}N&\to A\rtimes_{\alpha,\E_\Cor^N}N
\to (A\otimes C_0(G/N))\rtimes_{\alpha,\E_\Cor^N}N\\
&\to \M((A\otimes C_0(G/N))\rtimes_{\alpha\otimes \tau,\mu}G),
\end{split}
\end{equation}
in which the second homomorphism is given by functoriality for the generalised homomorphism
$A\to \M(A\otimes C_0(G/N)); a\mapsto a\otimes 1$. It is then easy to check on
the generators that the composition of the maps in (\ref{eq-composition}) coincides
with the faithful map (\ref{eq-map1}). Hence the first morphism
$A\rtimes_{\alpha,\mu|_N}N\to A\rtimes_{\alpha,\E_\Cor^N}N$ in (\ref{eq-composition}) must be faithful, too.
Hence the result.
\end{proof}

\section{Some questions}

There are actually many open questions related to the study of exotic  crossed products and, in particular,
the study of the minimal exact correspondence crossed-product functor $\rtimes_{\E^G_\Cor}$.
Many of those have been formulated and discussed in the papers \cite{Baum-Guentner-Willett} and \cite{Buss-Echterhoff-Willett}
and we therefore want to restrict here to questions related to restriction, extension, and induction of functors
from and to subgroups. The first question is quite obvious:

\begin{question}\label{Q1} Suppose $H$ is a closed subgroup of the locally compact group $G$. Is it always true that
the restriction $\rtimes_{\E^G_\Cor|_H}$ of the minimal exact correspondence functor $\rtimes_{\E^G_\Cor}$ for $G$
is equal to the minimal exact correspondence functor $\rtimes_{\E^N_\Cor}$ for $H$?
\end{question}

If the answer is ``yes'' then it follows from Proposition \ref{prop-BCrestrict} that the validity of the reformulated version
of the Baum-Connes conjecture due to Baum, Guentner, and Willett would pass from a group $G$ to any of its
closed subgroups. The previous section gives a positive answer if $H$ is normal in $G$.
At an earlier stage we thought that we have a positive answer at least for open subgroups $H$ of $G$,
but, unfortunately, we found a gap in our arguments. Any progress in this direction would be very appreciated
by the authors.

To formulate our next question, we start with a lemma.

\begin{lemma}\label{lem-decomp-action}
Suppose that $N\subseteq G$ is a closed normal subgroup of $G$ and that
$\alpha:G\to \Aut(A)$ is an action. Then there is an action
$\gamma^\E:G\to \Aut(A\rtimes_{\alpha,\E^N}N)$ (with $\rtimes_{\E^N}:=\rtimes_{\E^N_\Cor}$)
such that for each $g\in G$ the automorphism $\gamma^\E_g$ is given on the level
of $C_c(N,A)$ by the formula
$$ (\gamma^\E_g(f))(n)=\delta_g \alpha_g(f(g^{-1}ng)),$$
where $\delta_g$ denotes the module for the automorphism $C_g(n)=gng^{-1}$ of $N$.
\end{lemma}
\begin{proof}
It follows directly from Lemma \ref{lem-automorphisms2} applied to the automorphism
$\gamma_g=(\alpha_g, C_g)$ of $(A,N,\alpha)$  that the
automorphisms $\gamma^\E_g$ exist for all $g\in G$ and one easily checks that
$g\mapsto \gamma^\E_g$ is a homomorphism. Since for each $f\in C_c(N,A)$
the map $G\mapsto C_c(N,A); f\mapsto \gamma_g^\E(f)$ is continuous with respect to the
inductive limit topology on $C_c(N,A)$, the action $\gamma^E:G\to \Aut(A\rtimes_{\alpha,\E^N}N)$
is strongly continuous.
\end{proof}

\begin{remark} Note that there are analogous actions $\gamma^\un:G\to \Aut(A\rtimes_{\alpha,\un}N)$
and  $\gamma^\red:G\to \Aut(A\rtimes_{\alpha,\red}N)$ extending the action on $C_c(N,A)$ as
in the above lemma. These actions are known as the {\em decomposition actions} of
 $G$ on the respective crossed products by $N$. Indeed, if
  $i_N^\un: N\to U\M(A\rtimes_{\alpha,\un}N)$ and
 $i_N^\red: N\to U\M(A\rtimes_{\alpha,\red}N)$ are the canonical maps, then
 the pairs $(\gamma^\un, i_N^\un)$ and $(\gamma^\red, i_N^\red)$ are twisted actions
 in the sense of Green (\cite{Green:Local_twisted})
 of the pair $(G,N)$ on $A\rtimes_{\alpha,\un}N$ and $A\rtimes_{\alpha,\red}N$, respectively,
 and we get   canonical isomorphisms
 $$(A\rtimes_{\alpha,\un}N)\rtimes_{(\gamma^\un, i_N^\un),\un}(G,N)\cong A\rtimes_{\alpha,\un}G
 \;\text{and}\;
(A\rtimes_{\alpha,\red}N)\rtimes_{(\gamma^\red, i_N^\un),\red}(G,N)\cong A\rtimes_{\alpha,\red}G$$
(see \cite[Proposition 1]{Green:Local_twisted} for the full crossed products and
\cite[Proposition 5.2]{Kirchberg-Wassermann:Permanence} for the reduced case).
  Of course, if $i_N^\E:N\to U\M(A\rtimes_{\alpha,\E^N_\Cor}N)$ denotes the canonical map, then
  $(\gamma^\E, i_N^\E)$ is  also a Green-twisted action of $(G,N)$  on $A\rtimes_{\alpha,\E^N_\Cor}N$
  and we may wonder whether a similar decomposition isomorphism holds for the minimal
  exact correspondence functors. Note that by \cite[Theorem 1]{Echterhoff:Morita_twisted} we know that
  every twisted action of $(G,N)$ is Morita equivalent to an ordinary action of $G/N$.
 Using this it is not difficult to extend any strongly Morita compatible crossed-product functor for $G/N$ to the
  category of twisted $(G,N)$-algebras (e.g., see \cite[Definition 4.13]{Buss-Echterhoff:maximal-dual}).
    \end{remark}

  Thus it is natural to ask

\begin{question}\label{Q2}
Suppose $(A,\alpha)$ is a $G$-algebra and $N$ is a normal subgroup of $G$.
  Is there always  a canonical decomposition isomorphism
  $$(A\rtimes_{\alpha,\E^N_\Cor}N)\rtimes_{\gamma^\E, i_N^\E, \E^{G/N}_\Cor}(G,N)\congto
  A\rtimes_{\alpha, \E^G_\Cor}G?$$
  Is it at least true if $G=N\times H$ is the direct product of two closed subgroups?
  \end{question}

    Unfortunately we did not succeed in proving  such result even in the case where $G$ is a
  direct product $N\times H$.

Note that the candidate for
the isomorphism in Question \ref{Q2} can be described on the level of functions with compact supports,
but the precise description is a bit  tedious and we  refer to \cite[Proposition 1]{Green:Local_twisted} for the
details. In case where $G=N\rtimes H$ is a semi-direct product group, the twisted action of $G/N\cong H$
can be replaced by an untwisted crossed product by $H$ and the desired decomposition isomorphism should
then be given on the level of compactly supported continuous functions by
\begin{equation}\label{eq-Q2}
\Phi: C_c(N\times H,A)\to C_c(H, C_c(N,A));\; \Phi(f)(h)=f(\cdot, h).
\end{equation}

We believe that positive answers to Questions \ref{Q1} and \ref{Q2} would give the main steps
for  extending the permanence
results for the classical Baum-Connes conjecture as obtained in \cite{Chabert-Echterhoff:Permanence} to the
reformulated conjecture in the sense of Baum, Guentner, and Willett.
Note that we always have $\rtimes_{\E^N_\Cor}=\rtimes_{\E^G_\Cor|_N}$ by Theorem \ref{thm-ext-normal}.
  If $G=N\times H$ it follows also that $\rtimes_{\E^H_\Cor}=\rtimes_{\E^G_\Cor|_H}$. Thus in this case the
  above question may be generalised to

  \begin{question}\label{Q3}
Suppose $(A,\alpha)$ is an $N\times H$-algebra and $\rtimes_\mu$ is a crossed-product functor
for $G=N\times H$. Let $\Phi: C_c(N\times H,A)\to C_c(H, C_c(N,A))$ be given by $\Phi(f)(h)=f(\cdot, h)$.
Does $\Phi$ always extend to a \Star{}isomorphism
$$A\rtimes_{\alpha, \mu}(N\times H)\congto (A\rtimes_{\alpha^N, \mu|_N}N)\rtimes_{\alpha^H, \mu|_H}H?$$
\end{question}

The answer is positive for the full and reduced crossed products
but we know very little about the general case. Note that it follows from our Lemma \ref{lem-normal1} that
for any crossed-product functor $\rtimes_\mu$ for $G$, any  closed normal subgroup $N\subseteq G$, and
any $G$-algebra $(A,\alpha)$ there is a twisted action $(\gamma^\mu, i_N^{\mu|_N})$ of $(G,N)$
on $A\rtimes_{\alpha, \mu|_N}N$. Hence any crossed product functor $\rtimes_\nu$ for $G/N$
allows a decomposition crossed product $(A\rtimes_{\alpha,\mu|_N}N)\rtimes_{(\gamma^\mu, i_N^{\mu|_N}), \nu}(G,N)$.
Thus we may ask:

\begin{question}\label{Q4}
Given a crossed-product functor $\rtimes_\mu$ for $G$. Is there always an associated
 crossed-product functor $\rtimes_\nu$ for $G/N$ such that
 $$(A\rtimes_{\alpha,\mu|_N}N)\rtimes_{(\gamma^\mu, i_N^{\mu|_N}), \nu}(G,N)\cong A\rtimes_{\alpha,\mu}G?$$
 \end{question}

 It is actually not clear to us how to relate crossed-product functors for $G/N$ to crossed-product functors for $G$
 in a ``canonical'' way if $G$ does not decompose as a semi-direct product $N\rtimes H$.

\end{document}